\documentclass[reqno, 10pt]{amsart}

\usepackage[usenames,dvipsnames]{color}

\usepackage{amsthm,amsfonts,amssymb,amsmath,amsxtra} 
\usepackage[all]{xy}
 
\SelectTips{cm}{}
\usepackage{xr-hyper}
\usepackage[colorlinks=citecolor=Black, linkcolor=Red,
   urlcolor=Blue]{hyperref}
 
\usepackage{verbatim}

\makeatletter
\let\@wraptoccontribs\wraptoccontribs
\makeatother

\usepackage[margin=1.45in]{geometry}
 
\usepackage{mathrsfs}

\RequirePackage{xspace}
 
\RequirePackage{etoolbox}
 
\RequirePackage{varwidth}
 
\RequirePackage{enumitem}
 
\RequirePackage{tensor}
 
\RequirePackage{mathtools}
 
\RequirePackage{longtable}
 
\RequirePackage{multirow}

\usepackage{scalerel}

\newcommand\reallywidehat[1]{\arraycolsep=0pt\relax
\begin{array}{c}
\stretchto{
  \scaleto{
    \scalerel*[\widthof{\ensuremath{#1}}]{\kern-.5pt\bigwedge\kern-.5pt}
    {\rule[-\textheight/2]{1ex}{\textheight}} 
  }{\textheight}   
}{0.8ex}\\            
#1\\                 
\rule{-1ex}{0ex}
\end{array}
}

\newcommand{\sA}{\ensuremath{\mathscr{A}}\xspace}
\newcommand{\sB}{\ensuremath{\mathscr{B}}\xspace}

\newcommand{\sF}{\ensuremath{\mathscr{F}}\xspace}

\newcommand{\sR}{\ensuremath{\mathscr{R}}\xspace}

\newcommand{\fkG}{\ensuremath{\mathfrak{G}}\xspace}

\newcommand{\fkX}{\ensuremath{\mathfrak{X}}\xspace}

\newcommand{\mar}[1]{\marginpar{\tiny #1}}

\newcommand{\BC}{\ensuremath{\mathbb {C}}\xspace}

\newcommand{\BG}{\ensuremath{\mathbb {G}}\xspace}

\newcommand{\BQ}{\ensuremath{\mathbb {Q}}\xspace}
\newcommand{\BR}{\ensuremath{\mathbb {R}}\xspace}

\newcommand{\BZ}{\ensuremath{\mathbb {Z}}\xspace}

\newcommand{\CB}{\ensuremath{\mathcal {B}}\xspace}

\newcommand{\CG}{\ensuremath{\mathcal {G}}\xspace}

\newcommand{\CL}{\ensuremath{\mathcal {L}}\xspace}

\newcommand{\CO}{\ensuremath{\mathcal {O}}\xspace}
\newcommand{\CP}{\ensuremath{\mathcal {P}}\xspace}

\newcommand{\CV}{\ensuremath{\mathcal {V}}\xspace}

\newcommand{\RH}{\ensuremath{\mathrm {H}}\xspace}

\newcommand{\ad}{{\mathrm{ad}}}

\DeclareMathOperator{\Aut}{Aut}

\DeclareMathOperator{\Bun}{Bun}

\DeclareMathOperator{\diag}{diag}

\DeclareMathOperator{\Gal}{Gal}
\newcommand{\GL}{\mathrm{GL}}

\newcommand{\PGL}{{\mathrm{PGL}}}

\DeclareMathOperator{\Res}{Res}

\newcommand{\SL}{{\mathrm{SL}}}

\DeclareMathOperator{\Spec}{Spec\,}

\newcommand{\SU}{{\mathrm{SU}}}

\newcommand{\wt}{\widetilde}

\newtheorem{theorem}{Theorem}
\newtheorem{proposition}[theorem]{Proposition}
\newtheorem{prop/constr}[theorem]{Proposition/Construction}
\newtheorem{lemma}[theorem]{Lemma}

\newtheorem{corollary}[theorem]{Corollary}

\theoremstyle{definition}

\newtheorem{definition}[theorem]{Definition}
\newtheorem{example}[theorem]{Example}

\newtheorem{remark}[theorem]{Remark}

\newenvironment{altenumerate}
   {\begin{list}
      {\textup{(\theenumi)} }
      {\usecounter{enumi}
       \setlength{\labelwidth}{0pt}
       \setlength{\labelsep}{0pt}
       \setlength{\leftmargin}{0pt}
       \setlength{\itemsep}{\the\smallskipamount}
       \renewcommand{\theenumi}{\roman{enumi}}
      }}
   {\end{list}}

\numberwithin{equation}{section}
\numberwithin{theorem}{section}

\setitemize[0]{leftmargin=0.2in,itemsep=\the\smallskipamount}
\setenumerate[0]{leftmargin=0.2in,itemsep=\the\smallskipamount}

\renewcommand{\to}{%
   \ifbool{@display}{\longrightarrow}{\rightarrow}%
   }
\let\shortmapsto\mapsto
\renewcommand{\mapsto}{%
   \ifbool{@display}{\longmapsto}{\shortmapsto}%
   }
\newlength{\olen}
\newlength{\ulen}
\newlength{\xlen}
\newcommand{\xra}[2][]{%
   \ifbool{@display}%
      {\settowidth{\olen}{$\overset{#2}{\longrightarrow}$}%
       \settowidth{\ulen}{$\underset{#1}{\longrightarrow}$}%
       \settowidth{\xlen}{$\xrightarrow[#1]{#2}$}%
       \ifdimgreater{\olen}{\xlen}%
          {\underset{#1}{\overset{#2}{\longrightarrow}}}%
          {\ifdimgreater{\ulen}{\xlen}%
             {\underset{#1}{\overset{#2}{\longrightarrow}}}
             {\xrightarrow[#1]{#2}}}}%
      {\xrightarrow[#1]{#2}}
   }
\makeatother
\newcommand{\xyra}[2][]{%
   \settowidth{\xlen}{$\xrightarrow[#1]{#2}$}%
   \ifbool{@display}%
      {\settowidth{\olen}{$\overset{#2}{\longrightarrow}$}%
       \settowidth{\ulen}{$\underset{#1}{\longrightarrow}$}%
       \ifdimgreater{\olen}{\xlen}%
          {\mathrel{\xymatrix@M=.12ex@C=3.2ex{\ar[r]^-{#2}_-{#1} &}}}%
          {\ifdimgreater{\ulen}{\xlen}%
             {\mathrel{\xymatrix@M=.12ex@C=3.2ex{\ar[r]^-{#2}_-{#1} &}}}
             {\mathrel{\xymatrix@M=.12ex@C=\the\xlen{\ar[r]^-{#2}_-{#1} &}}}}}%
      {\mathrel{\xymatrix@M=.12ex@C=\the\xlen{\ar[r]^-{#2}_-{#1} &}}}%
   }
\makeatletter
\newcommand{\xla}[2][]{%
   \ifbool{@display}%
      {\settowidth{\olen}{$\overset{#2}{\longleftarrow}$}%
       \settowidth{\ulen}{$\underset{#1}{\longleftarrow}$}%
       \settowidth{\xlen}{$\xleftarrow[#1]{#2}$}%
       \ifdimgreater{\olen}{\xlen}%
          {\underset{#1}{\overset{#2}{\longleftarrow}}}%
          {\ifdimgreater{\ulen}{\xlen}%
             {\underset{#1}{\overset{#2}{\longleftarrow}}}
             {\xleftarrow[#1]{#2}}}}%
      {\xleftarrow[#1]{#2}}
   }
\newcommand{\isoarrow}{%
   \ifbool{@display}{\overset{\sim}{\longrightarrow}}{\xrightarrow\sim}%
   }

\newcommand{\quash}[1]{}

 \newcommand{\ome}{{\bf a}}
 
 \DeclareFontEncoding{OT2}{}{} 
  \newcommand{\textcyr}[1]{%
    {\fontencoding{OT2}\fontfamily{wncyr}\fontseries{m}\fontshape{n}%
     \selectfont #1}}
\newcommand{\Sha}{{\mbox{\textcyr{Sh}}}}

\begin{document}

\title{On tamely ramified  $\CG$-bundles on curves}

\author[G. Pappas and M. Rapoport]{Georgios Pappas \and Michael Rapoport  \\ \\ { {\scriptsize with an appendix by}} Brian Conrad}

\address{GP: Dept. of Mathematics, Michigan State University, E. Lansing, MI 48824, USA}
\email{pappasg@msu.edu}

\address{MR: Mathematisches Institut der Universit\"at Bonn, Endenicher Allee 60, 53115 Bonn, Germany, and University of Maryland, Department of Mathematics, College Park, MD 20742, USA}
\email{rapoport@math.uni-bonn.de}

\address{BC: Department of Mathematics, Stanford University, Stanford, CA 94305, USA}
\email{conrad@math.stanford.edu}

\thanks{{\sl 2020 Mathematics Subject Classification:} 20G35, 14L15 (primary), 20J06, 14D20 (secondary) }
\thanks{{\sl Keywords:} Bruhat-Tits group schemes over curves, local types of $(G,
\Gamma)$-bundles, purity for $(G, \Gamma)$-bundles, Hasse principle.}
\thanks{GP was supported by NSF grant \#DMS-2100743.}

\date{\today}
\maketitle

\begin{abstract}
 We consider parahoric Bruhat-Tits group schemes over a smooth projective curve and torsors under them.  If the characteristic of the ground field is either zero or positive but not too small and the generic fiber is absolutely simple and simply-connected, we show that such group schemes  can be written as invariants of reductive group schemes over a tame cover of the curve. We relate the torsors under the Bruhat-Tits group scheme and torsors under the reductive group scheme over the cover which are equivariant for the action of the covering group. For this, we develop a theory of local types for such equivariant torsors. We also relate the moduli stacks of  torsors under the Bruhat-Tits group scheme and equivariant torsors under the reductive group scheme  over the cover. 
\end{abstract}

\tableofcontents

\section{Introduction}
$\CG$-bundles on a smooth projective curve over a field $k$ were first considered in \cite{PRq}. Here $\CG$ is a \emph{parahoric Bruhat-Tits group scheme} over $X$, i.e. a smooth group scheme with reductive generic fiber and such that all fibers over closed points are parahoric group schemes in the sense of Bruhat-Tits (the terminology is due to J.~Heinloth \cite{Hein}).  The most obvious examples are given by reductive group schemes over the curve. For various reasons, it is useful to have a concrete description of $\CG$ and of its torsors. One reason is to give a Verlinde formula for the stack of $\CG$-bundles, as conjectured in \cite{PRq}, cf. \cite{Dam1, HoKu, Muk}.  

The first aim of the present paper is to give a concrete description of  general parahoric Bruhat-Tits group schemes. This is done in terms of reductive group schemes over Galois coverings of the curve, and is inspired by a paper by Balaji-Seshadri \cite{BalaS}.  In \cite{BalaS}, Bruhat-Tits group schemes are considered which are of the form $\CG=(\Res_{X'/X}(G_{X'}))^\Gamma$, where $X'/X$ is a Galois cover of curves with Galois group $\Gamma$ and where $G_{X'}$ is the constant group scheme over $X'$ corresponding to a semi-simple simply connected group $G$ over $k$, equipped with an action of $\Gamma$, cf. also \cite{Dam, HoKu}. In loc.~cit., it is assumed that the ground field $k$ has characteristic zero. One of our main results is that, at least if  $k$ is algebraically closed and excluding small positive characteristics  (for classical groups it is sufficient to exclude $p=2$), any parahoric Bruhat-Tits group scheme $\CG$ over a curve $X$ with generic fiber an absolutely  simple simply connected group is of the form  $\CG=(\Res_{X'/X}(\CG'))^\Gamma$, where $X'/X$ is a \emph{tame} Galois covering and where $\CG'$ is a reductive group scheme over $X'$. This class is strictly larger than the class considered in \cite{BalaS} or in \cite{HoKu} since, for example, here $\CG'$ is not necessarily a constant group scheme. 

Once Bruhat-Tits group schemes are presented in this way (by reductive group schemes over coverings), it becomes natural to try to describe related concepts in terms of such a presentation. The second aim of this paper is to understand $\CG$-torsors in terms of  this presentation of $\CG$. It may at first glance seem reasonable to guess that the functor $\CP'\mapsto \CP=(\Res_{X'/X}(\CP'))^\Gamma$ on the category of $\CG'$-bundles on $X'$ with compatible $\Gamma$-action yields a $\CG$-bundle and even gives an equivalence of categories with the category of $\CG$-bundles on $X$. However, this turns out not to be the case, not even when $\CG'=G_{X'}$, as was pointed out by Damiolini \cite{Dam}. In \cite{Dam} it is explained (in the context of \cite{BalaS}) how to understand these ``$(\CG', \Gamma)$-bundles'' in terms of torsors under parahoric Bruhat-Tits group schemes   which are ``twisted versions" of $\CG$. Furthermore, Damiolini introduces local obstructions (``local types of $(\CG', \Gamma)$-bundles'') which characterize those $(\CG', \Gamma)$-bundles having the same associated twisted parahoric Bruhat-Tits group schemes.  Here, we will give a concrete cohomological expression for Damiolini's local types, see also \cite{DH}.

The third aim of this paper is to establish corresponding facts for the moduli stacks of $\CG$-bundles, resp. $(\CG', \Gamma)$-bundles. It turns out that   $(\CG', \Gamma)$-bundles  on $X'$ with  a fixed local type $\tau$ can be related to torsors for a ``twisted" parahoric Bruhat-Tits group scheme $^\tau\!\CG$ over the base curve $X$. This is  interesting even when    $\CG'=G\times_k X'$ is  constant over $X'$ and we take $\Gamma$ to act trivially on $G$, in which case $\CG=G\times_k X$. In fact, Balaji-Seshadri \cite{BalaS} mostly consider this particular special case.

  The motivation given above for considering $(\CG', \Gamma)$-bundles on $X'$ is in a sense  askew to the historical development.  
 Indeed, in Grothendieck's Bourbaki talk \cite{GrothBour} on Weil's \emph{M\'emoire} \cite{Weil}, Grothendieck starts with an analytic space $X'$ (over $\BC$) which is equipped with a faithful and properly discontinuous action of a discrete group $\Gamma'$ (with finite stabilizer groups), and considers \emph{analytic $(G, \Gamma')$-bundles} on $X'$, where $G$ is a complex Lie group. Let $X=X'/\Gamma$. When the action of $\Gamma$ is free, then Grothendieck notes that there is an equivalence between $(G, \Gamma')$-bundles on $X'$ and $G$-bundles on $X$. In general, he gives a cohomological presentation of the set of isomorphism classes of $(G, \Gamma')$-bundles on $X'$ (\cite[\S 1, no. 3]{GrothBour}),  and he identifies among the set of $(G, \Gamma')$-bundles on $X'$ those that correspond to \emph{$(G, \Gamma')$-divisors} in the sense of Weil \cite{Weil} on $X'$ (these are the generalizations to arbitrary $G$ of the matrix divisors in the case of $G=\GL_n$, which  can be identified with vector bundles in the case when $X'$ is a projective variety (\cite[\S2, no. 3]{GrothBour})).
 He also gives a \emph{local classification} of $(G, \Gamma')$-bundles on $X'$ at $x'\in X'$ in terms of $\RH^1(\Gamma_{x'}, G)$,  when $X'$ is smooth of dimension one (\cite[\S 2, Prop. 1]{GrothBour}). Grothendieck remarks in \cite[\S2, no.6]{GrothBour} that one can express  $(G, \Gamma')$-bundles on $X'$ purely in terms of $X$ and the local data at the points of the branch locus of the map $X'\to X$, provided that $X'$ is a smooth algebraic curve. Grothendieck stops short of providing a definitive result. This may be related to the fact that in our set-up, even when     $\CG'=G\times_k X'$ is  constant over $X'$ and we take $\Gamma$ to act trivially on $G$, we need the twisted forms $^\tau\!\CG$ over  $X$ to accommodate all $(G, \Gamma')$-bundles on $X'$. In other words,   the results,  both of this paper and also of Balaji-Seshadri \cite{BalaS} described in the previous paragraph, may be considered as making precise these remarks of Grothendieck. 
 
 The motivation of Weil and Grothendieck  is geometric  (in fact, they work over $\BC$ as a base field), as is that of Balaji-Seshadri, Damiolini and Hong.
 However, there is also motivation from arithmetic. Indeed,  a similar situation, in which $X'/X$ is replaced by a tamely ramified cover of spectra of power series rings,  arises in the study of the moduli and deformation spaces of local Galois representations with fixed inertia types as in the work of Caraiani-Levin \cite{CaraLev}, and Caraiani-Emerton-Gee-Savitt \cite{Caraetc}. In these papers, the  inertia type determines  the analogue of the local type of a corresponding Breuil-Kisin module and thus prescribes a parahoric group scheme and a torsor under it. As a further link between the geometric and the arithmetic situation, we also note that the proof of the purity theorem for $(\CG', \Gamma)$-bundles here is inspired by our proof  in the tame case \cite{PRglsv} of Ansch\"utz' extension theorem on $\CG$-torsors on the punctured spectrum of $A_{\rm inf}$ \cite{An}.

Let us now formulate our main results, making the simplifying assumption that $G$ is semi-simple and simply connected and that the ground field $k$ is algebraically closed (in the body of the text, we are more specific as to our precise  assumptions). The first result concerns the structure of Bruhat-Tits group schemes, cf. Theorem \ref{descred}.

\begin{theorem}\label{intro-descred}
Let $\CG$ be a parahoric Bruhat-Tits group scheme  over $X$ such that its generic fiber $G$ is (simply connected and)  absolutely  simple. Assume  $p\neq 2$ if $G$ is a classical group, and $p\neq 2,3$ if $G$ is not of type $E_8$, and $p\neq 2, 3, 5$ if $G$ is of type $E_8$.  Then there exists a tame Galois covering $X'/X$ with Galois group $\Gamma$ and a reductive group scheme $\CG'$ over $X'$ with generic fiber $G'=G\otimes_{k(X)} k(X')$  to which  the $\Gamma$-action on $G'$ extends such that $\CG=\Res_{X'/X}(\CG')^\Gamma$.  
\end{theorem}

 The proof is based on an amusing analogous local result in Bruhat-Tits theory (Proposition \ref{BThyper}) which represents a parahoric subgroup as the set of rational points of a hyperspecial parahoric subgroup for a finite extension of the local field. Similar results  in characteristic zero, based on a local result in a loop group context, also appear in work of Damiolini and Hong \cite{DH}.

The next result concerns the description of $\CG$-bundles in terms of the presentation of $\CG$ in the form of Theorem \ref{intro-descred}. More precisely, for the next results let  $X'/X$ be a tame Galois covering $X'/X$ with Galois group $\Gamma$ and let $\CG'$ be a parahoric Bruhat-Tits group scheme over $X'$  with generic fiber $G'=G\otimes_{k(X)} k(X')$  to which  the $\Gamma$-action on $G'$ extends such that $\CG=\Res_{X'/X}(\CG')^\Gamma$. Any $\CG$-bundle defines a  $(\CG', \Gamma)$-bundle, i.e.  a $\CG'$-bundle with a $\Gamma$-action compatible with the $\Gamma$-action on $\CG'$, and this association defines a fully faithful functor, cf. Proposition \ref{BTtors}. The failure of essential surjectivity is described by the following result, cf. Proposition \ref{propLeray} and Proposition \ref{LTfiber}.

\begin{theorem}\label{intro-exseq}
There is an exact sequence of pointed sets, where in the middle is the set of isomorphism classes of $(\CG', \Gamma)$-bundles on $X'$,
 \begin{equation}\label{intro-eqleray4}
0\to \RH^1(X, \CG)\to \RH^1(X'; \Gamma, \CG' )\xrightarrow{\  {\rm LT}\ }   \prod_{x\in B}  \RH^1(\Gamma_{x'},  \bar\CG'^{\rm red}_{x'}(k)) .
\end{equation}
Here, $B\subset X$ is the (finite) branch locus of $\pi: X'\to X$ and, for $x\in B$ a lift $x'$ of $x$ to $X'$ is chosen;  $\bar\CG'^{\rm red}_{x'}$ denotes the maximal reductive quotient of the special fiber of $\CG'$ over $x'$, and $\Gamma_{x'}$ is the stabilizer of $x'$ in $\Gamma$. 

Furthermore, the map ${\rm LT}$ is surjective and the fibre of ${\rm LT}$ over $\tau\in  \prod_{x\in B}  \RH^1(\Gamma_{x'},  \bar\CG'^{\rm red}_{x'}(k))$ can be identified with $ \RH^1(X, ^\tau\!\CG)$. Here $ ^\tau\!\CG$ is a parahoric Bruhat-Tits group scheme over $X$  with generic fiber $G$ whose localizations  $^\tau\!\CG_x=^\tau\!\CG\times_X \Spec(O_x)$ at the completed local rings are  isomorphic to  $\CG_x$ at all points $x$ outside $B$ and which   can be described explicitly in terms of $\tau$, for all $x\in B$.
\end{theorem}

There are two possible approaches to the explicit determination of the cohomology set $\RH^1(\Gamma_{x'},  \bar\CG'^{\rm red}_{x'}(k))$ occurring in \eqref{intro-eqleray4}.  This is the set of ``local types" at $x'$.   The first approach also appears in \cite{DH} (with a different proof), and both approaches also appear under special hypotheses in \cite{BalaS}.
 The first approach  is via a $\Gamma_{x'}$-stable maximal torus $T$ of $\bar\CG'^{\rm red}_{x'}$, with Weyl group $W$, and uses the bijection
\[
\RH^1(\Gamma_{x'}, \bar\CG'^{\rm red}_{x'}(k))=\big[\ker({\rm N}_{\Gamma_{x'}}: T(k)\to T(k))/I_{\Gamma_{x'}}T(k)\big]/{W^{\Gamma_{x'}}},
\]
cf. \eqref{H1Stein} and \eqref{H1torus}. Here ${\rm N}_{\Gamma_{x'}}=\sum_{\gamma\in\Gamma_{x'}} \gamma$ is the norm element and $I_{\Gamma_{x'}}$ the augmentation ideal in the group algebra $\BZ[\Gamma_{x'}]$.  

The second approach is via the Bruhat-Tits building $\sB(G, K_x)$.  Namely, we establish a bijection
\[
 \RH^1(\Gamma_{x'},  \bar\CG'^{\rm red}_{x'}(k))=G(K_x)\backslash \big(\sB(G, K_x)\cap (G(K'_{x'})\cdot {\ome}_x)\big),
\]
cf. Proposition \ref{LTprop2}.
Here $\CG(O_x)$ is the stabilizer of the point $\ome_x\in\sB(G, K_x)$, and $G(K'_{x'})\cdot {\ome}_x$ denotes the $G(K'_{x'})$-orbit of $\ome_x$, where we use the natural embedding $\sB(G, K_x)\subset \sB(G, K'_{x'})$.

Our final result concerns the moduli stacks $\Bun_{\CG}$ of $\CG$-bundles, resp. $\Bun_{(\CG', \Gamma)}$ of $(\CG', \Gamma)$-bundles, cf. Theorem \ref{thmlocconst}. These are smooth algebraic stacks over $k$. 
\begin{theorem}\label{intro-thmlocconst}

\noindent a) The local type map
\[
{\rm LT}\colon  \Bun_{(\CG', \Gamma)}(k)\to \prod_{x\in B} \RH^1({\Gamma_{x'}}, \bar\CG'^{\rm red}_{x'}(k))
\] is locally constant.

\noindent b) For each element $\tau\in   \prod_{x\in B} \RH^1({\Gamma_{x'}}, \bar\CG'^{\rm red}_{x'}(k))$, let $(\Bun_{(\CG', \Gamma)})_\tau$ be the open and closed substack where the value of ${\rm LT}$ equals $\tau$. Fix a point in $(\Bun_{(\CG', \Gamma)})_\tau(k)$. Then there is a canonical isomorphism of algebraic stacks preserving the base points
\[
(\Bun_{(\CG', \Gamma)})_\tau\simeq \Bun_{\, ^\tau\!\CG} .
\]
Here $^\tau\!\CG$ is the parahoric Bruhat-Tits group scheme from Theorem \ref{intro-exseq}. Hence
\[
\Bun_{(\CG', \Gamma)}\simeq \bigsqcup_{\tau\in    \prod_{x\in B} \RH^1({\Gamma_{x'}}, \bar\CG'^{\rm red}_{x'}(k))}\Bun_{\,^\tau\!\CG} .
\]
\end{theorem}
This theorem implies, by Heinloth's connectedness theorem \cite{Hein}, that 
\[\pi_0(\Bun_{(\CG', \Gamma)})=\prod_{x\in B} \RH^1({\Gamma_{x'}}, \bar\CG'^{\rm red}_{x'}(k)).
\] 

 The lay-out of the paper is as follows. In section \ref{s:tameBT} we introduce Bruhat-Tits group schemes and prove Theorem \ref{intro-descred}. In section \ref{s:tamebdls}, we discuss the key concept of $(\CG', \Gamma)$-bundles and its relation to $\CG$-bundles. In section \ref{s:purity}, we prove the purity theorem, Theorem \ref{thmpur},  which is the basis of the local constancy statement in Theorem \ref{intro-thmlocconst}. A key argument in the proof of Theorem \ref{thmpur} is due to Scholze. In section \ref{s:localt} we discuss the local analogue of $(\CG', \Gamma)$-bundles and give their classification.  In section \ref{s:locglob},  these local results are used to prove Theorem \ref{intro-exseq}. We also give a variant over a finite base field $k$. In section \ref{s:moduli} we prove Theorem \ref{intro-thmlocconst}. In the final section \ref{s:concl} we speculate how one could possibly express other objects associated to parahoric Bruhat-Tits group schemes over $X$ in terms of a presentation by a reductive group scheme over a Galois cover of $X$.

 Finally, in the 
  appendix,  B. Conrad gives a proof of the  Hasse principle for adjoint groups over function fields with finite field of constants.

 We thank   J.~Heinloth, J.~Hong, G. Prasad and P.~Scholze for helpful remarks.   We also thank B.  Conrad for providing the appendix. The second author presented the results of this paper at the conference ``Bundles and conformal blocks with a twist'' in June 2022 in Edinburgh. He thanks the organizers for the invitation. We also thank the referee for helpful remarks.

\smallskip

{\bf Notations.} Let $X$ be a curve, i.e., a geometrically connected smooth projective scheme of dimension one over a perfect field $k$. We denote by $k(X)$ or simply $K$ its function field. For a closed point $x$, we denote by $O_{X, x}$ its local ring and by $k(x)$ its residue field. We denote by $O_x$ the completion of $O_{X, x}$, by $K_x$ its fraction field, 
and by $O^h_{X, x}$, resp. $O^{sh}_{X, x}$, its henselization, resp. strict henselization.  If $\bar x$ is a geometric point over $x$, we denote by $O_{\bar x}$ the strict completion of $O_{X, x}$, a complete local ring with residue field $k(\bar x)$. We also write $K_{\bar x}$ for its fraction field. If $\CG$ is a scheme over $X$ (resp. over $\Spec(K)$), we write $\CG_x$ for $\CG\times_X\Spec(O_x)$ (resp. for $\CG\times_{\Spec(K)}\Spec(K_x)$).

\section{Tamely ramified Bruhat-Tits group schemes}\label{s:tameBT}
\subsection{Bruhat-Tits group schemes}
 
A \emph{Bruhat-Tits (BT) group scheme} over the curve $X$ is a smooth group scheme $\CG$ over $X$ such that its generic fiber $G=\CG\times_X \Spec(k(X))$ is reductive and such that, for every closed point $x\in X$,  the base change $\CG\times_X\Spec(O_x)$ is a Bruhat-Tits group scheme   over the completion $O_x$ of the local ring $O_{X, x}$ for $G\times_X\Spec(K_x)$.   This last condition means that $\CG\times_X\Spec(O_x)$ is a quasi-parahoric group scheme over $O_x$, i.e., $\CG\times_X\Spec(O_x)$ is a smooth affine group scheme such that $\CG(O_{\bar x})$ is a quasi-parahoric subgroup of  $G(K_{\bar x})$. Recall that a quasi-parahoric subgroup of  $G( K_{\bar x})$ is a bounded subgroup that contains a parahoric subgroup with finite index, cf. \cite{PRintls}.  A \emph{parahoric Bruhat-Tits group scheme} over $X$ is the neutral component of a Bruhat-Tits group scheme. Equivalently, it is a smooth group scheme $\CG$ over $X$ such that its generic fiber $G=\CG\times_X \Spec(k(X))$ is reductive and such that, for every closed point $x\in X$,  the base change $\CG\times_X\Spec(O_x)$ is a parahoric Bruhat-Tits group scheme.  
We also say that $\CG$ is a (parahoric) BT group scheme for $G$.

\begin{remark}\label{stabBT}
 \begin{altenumerate}
\item A particular kind of quasi-parahoric group scheme arises as the stabilizer of  a point ${\bf a}_{x}$ in the extended Bruhat-Tits building $ \sB^e(G\otimes_{k(X)} K_x, K_x)$, i.e.,  
$\CG (O_{\bar x})$ is the  stabilizer  in $G(K_{\bar x})$ of  ${\bf a}_{ x}$. A \emph{stabilizer Bruhat-Tits group scheme} over $X$ is a BT group scheme such that $\CG\times_X\Spec(O_x)$ is the stabilizer Bruhat-Tits group scheme corresponding to a point ${\bf a}_{x}$ in the extended Bruhat-Tits building $ \sB^e(G\otimes_{k(X)} K_x, K_x)$, for every $x\in X$.  Note that in general a parahoric group scheme is not necessarily of this form. In the sequel, we also write simply $ \sB^e(G, K_x)$ for $ \sB^e(G\otimes_{k(X)} K_x, K_x)$.
\item If $G$ is simply connected, then any quasi-parahoric group scheme is automatically connected, so it is a  parahoric BT group scheme and a stabilizer BT group scheme. 
\item Any reductive group scheme  over $X$ is a  parahoric BT group scheme and a stabilizer BT group scheme.  Indeed, this is a statement of Bruhat-Tits theory of group schemes over $O_x$. After an \'etale extension, we may assume that we are concerned with a split group. Then the assertion follows from \cite[4.6.22]{BTII}.  
\end{altenumerate} 

We decided to put the emphasis on stabilizer BT group schemes since they seem well-adapted to our way of giving a concrete presentation of BT group schemes over $X$, cf. Introduction. 
\end{remark}

\subsection{BT group schemes and coverings}\label{ss:BTcov}
Let $X'\to X$ be a  Galois cover, with finite Galois group $\Gamma$.  In other words, $X'$ is the normalization of $X$ in a finite (separable) Galois extension $K'$ of $k(X)$ with Galois group $\Gamma$. We allow the  field of constants of $X'$ to be a finite extension of $k$ (hence $X'$ is a curve in our sense over a finite extension $k'$ of $k$). For $x'\in X'$, we denote by $\pi(x')$, or simply $x$, its image in $X$. We denote by $\Gamma_{x'}\subset \Gamma$ the inertia subgroup at $x'$. Let $G$ be a reductive group over $k(X)$ and set  $G'=G\otimes_{k(X)} k(X')$. There are two constructions of BT group schemes for $G$, resp. $G'$.

\begin{altenumerate}
\item Let $\CG$ be a stabilizer  BT group scheme for $G$ over $X$. The base change group $G'$ acts on the extended BT building $\sB^e(G, K'_{x'})$, for every $x'\in X'$. Indeed, for this we identify  $\sB^e(G, K'_{x'})$ with  $\sB^e(G', K'_{x'})$. Furthermore, for any $x'\in X'$, the base change $\CG_x\otimes_{O_x} O'_{x'}$ fixes the point ${\bf a}_{ x}\in\sB^e(G, K_{x})$ corresponding to the  group $\CG_{ x}$. Here we consider ${\bf a}_{ x}$ as a point of  $\sB^e(G, K'_{x'})$, using the natural embedding $\sB^e(G, K_{x})\to\sB^e(G, K'_{x'})$. Hence, denoting by $\CG'_{{\bf a}_{ x}}$ the stabilizer group scheme for $G'\otimes_{k(X)}K'_{x'}$ defined by the point  ${\bf a}_{ x}\in\sB^e(G', K'_{x'})$, we obtain  a morphism extending the identity in the generic fibers,
 \begin{equation}\label{mapupst}
 \CG_{{\bf a}_{x}}\otimes_{O_x} O'_{x'}\to\CG'_{{\bf a}_{  x}} . 
 \end{equation} This morphism is an isomorphism when $x'$ is outside the ramification locus $R'$ of $X'\to X$. Hence, by Beauville-Laszlo glueing along the finitely many points of $R'$, there exists a  unique smooth group scheme $\CG'$ over $X'$ such that $\CG'_{|(X'\setminus R')}=\CG\times_X (X'\setminus R')$ and  such that the map $\CG\times_X\Spec(O'_{x'})\to\CG'\times_{X'}\Spec(O'_{x'})$ is identified with the  map \eqref{mapupst},  for every $x'\in X'$. Hence $\CG'$ is a stabilizer BT group scheme for $G'$ over $X'$.

\item Conversely,  assume that the Galois covering $X'/X$ is tame, i.e. all the inertia subgroups $\Gamma_{x'}$ of $\Gamma$ are of order prime to ${\rm char}\, k$. Let  $\CG'$ be a stabilizer BT group scheme for $G'$  on $X'$ obtained from a  collection of points ${\bf a}_{x}$ in the buildings for $G\otimes_K K_x$. Then $\CG'$ has  a $\Gamma$-action, inducing in the generic fiber the given $\Gamma$-action on $G'$. Then, using \cite[(1.7.6)]{BTII},
we see that $\CG=\Res_{X'/X}(\CG')^\Gamma$ is a stabilizer BT group scheme for $G$ over $X$ for the points ${\bf a}_x$.  Indeed, the generic fiber of $\CG$ is identified with $G$, the $\Spec(O_{\bar x})$-points of $\CG$ give the correct stabilizer subgroups for all $x$ and the smoothness of $\CG$ follows from Edixhoven's lemma \cite{Edix}. 
\end{altenumerate}
\begin{definition}\label{defrelated}
Let $X'/X$ be a tame Galois cover with Galois group $\Gamma$. Let $G$ be a reductive group over $K$ and $G'$ its base change over $K'$. Let $\CG$, resp. $\CG'$ be stabilizer BT group schemes for $G$, resp. $G'$ over $X$, resp. $X'$. Then $\CG$ and $\CG'$ are said to be \emph{$\Gamma$-related} if both are the stabilizer groups related to the same collection of points $\{{\bf a}_x\in\sB^e(G, K_{x})\mid x\in X\}$. In this case $\CG=\Res_{X'/X}(\CG')^\Gamma$. 
\end{definition}

\begin{remark}Note that for any stabilizer BT group scheme $\CG$ for $G$, there is a unique stabilizer BT group scheme $\CG'$ for $G'$ such that $\CG$ and $\CG'$ are $\Gamma$-related. On the other hand, a $\Gamma$-invariant stabilizer BT group $\CG'$ for $G'$  is not in general part of a pair $(\CG, \CG')$ of $\Gamma$-related  stabilizer BT group schemes  (although this does hold if $G$ is semi-simple simply connected).  This phenomenon is illustrated by the example of the projective linear group $G$ of a quaternion algebra over a local field $F$: after base change by a quadratic extension $F'/F$, the group $G'=G\otimes_F F'$ is isomorphic to $\PGL_2$. Then both the unique $\Gamma$-invariant Iwahori subgroup of $G'$ and its normalizer are $\Gamma$-invariant stabilizer groups but only the latter is the stabilizer of a $\Gamma$-invariant point in the building and hence is part of a $\Gamma$-related pair of stabilizer groups. 
\end{remark}

\begin{proposition}\label{BTgps}
 Let $X'/X$ be a tame Galois cover with Galois group $\Gamma$, and $G$ and $G'$ as above. 

a)  Any parahoric  group scheme $\CG$ over $X$ is of the form $(\Res_{X'/X}(\CG')^\Gamma)^\circ$, where  $\CG'$ is the  stabilizer Bruhat-Tits group scheme for $G'$ over $X'$ which is $\Gamma$-related to a stabilizer group scheme $\CG^\sharp$ with $\CG$ as its connected component.  Here the parahoric  group scheme $(\CG')^\circ$ with its $\Gamma$-action is uniquely  determined from $\CG$. 

b) Assume that $G$ or, equivalently,  $G'$ is semi-simple simply connected. Then, starting with $\CG$  and applying the procedure in i) to obtain the parahoric BT group scheme $\CG'$ with $\Gamma$-action, and then applying procedure ii) to $\CG'$ with its $\Gamma$-action gives back $\CG $.  Conversely, starting with a collection of points $\{{\bf a}_x\in\sB^e(G, K_{x})\mid x\in X\}$ and a parahoric BT group scheme $\CG'$ with $\Gamma$-action such that $\CG'_{x'}$ is the stabilizer group for ${\bf a}_{\pi(x)}\in \sB(G, K'_{x'})$, then, applying the procedure ii)  gives  a parahoric BT group scheme $\CG$ which, applying procedure i), gives back $\CG'$ with its $\Gamma$-action.

\end{proposition} 
\begin{proof}  
Follows from the discussion above.
\end{proof}

\begin{remark}\label{RemarkReductive} In the above set-up, if $\CG$ is reductive and $\CG$ and $\CG'$ are $\Gamma$-related, then $\CG'$ is also reductive. This follows from the fact that all hyperspecial points of $\sB^e(G, K_{x})$ remain hyperspecial in $\sB^e(G, K'_{x'})$.
\end{remark}

\begin{theorem}\label{descred}
  Assume that $G$ is absolutely simple and simply connected. Also assume  $p\neq 2$ if $G$ is a classical group\footnote{Note that ``$G$ classical'' excludes the trialitarian forms of $D_4$.},  $p\neq 2,3$ if $G$ is not of type   $E_8$, and $p\neq 2, 3, 5$ if $G$ is of type $E_8$.  Let $\CG$ be a Bruhat-Tits group scheme for $G$ over $X$. After a finite extension of the base field $k$, there exists a tame Galois covering $X'/X$ of curves over $k$ with Galois group $\Gamma$ and a reductive group scheme $\CG'$ over $X'$  such that $\CG$ and $\CG'$ are $\Gamma$-related.
\end{theorem}

This is a global consequence of the following statement of Bruhat-Tits theory which seems new (related statements appear in work of M. Larsen \cite[Lem. 2.4]{Lar} and P. Gille \cite[Lem. 2.2]{Gi}, see also \cite[Lem. 4.3]{Cotner}).

\begin{proposition}\label{BThyper}
Let $G$ be a semi-simple simply connected group over a local field\footnote{By a local field we mean a field which is complete for a discrete valuation, with perfect residue field.} $F$, and let $\CG$ be a parahoric group scheme for $G$ over $O_F$. Then there exists a finite extension $F'/F$ such that $G'=G\otimes_F F'$ is split, and  a hyperspecial parahoric group scheme $\CG'$ for  $G'$ such that $\CG(O_F)=G(F)\cap \CG'(O_{F'})$ and $\CG(O_{\breve F})=G({\breve  F})\cap \CG'(O_{{\breve F}'})$. Here $\breve F$, resp. $\breve F'$, denotes the completion of the maximal unramified extension of $F$, resp. $F'$ (in a common algebraic closure).

 Assume  $p\neq 2$ if $G$ is a classical group, and $p\neq 2,3$ if $G$ admits no  factor of type $E_8$, and $p\neq 2, 3, 5$ if $G$ admits a factor of type $E_8$. Then $F'$ can be chosen to be tamely ramified, provided that $G$ is \quash{absolutely  simple or $G$ is}tamely ramified.
 \end{proposition} 

\begin{proof}(following an e-mail of G.~Prasad) We may assume that $G$ is  simple.
Let $K$ be a finite unramified extension  of $F$ such that the relative rank of $G$  does not increase if we make a further unramified extension of $K$. Then $G$ is quasi-split over $K$ since by Steinberg's theorem it is quasi-split over the maximal unramified extension of $F$. Let $\Gamma$ denote the Galois group of $K/F$. The building of $G(F)$ is the fixed point set under the  action of $\Gamma$ on   the building of $G(K)$. Now every vertex  in the building of $G(F)$ is just the barycenter of a minimal $\Gamma$-stable facet of the building of $G(K)$, cf. \cite[Thm. 5.1.25]{BTII}. Therefore, its coordinates in terms of a basis of the affine root system (determined by an alcove in an apartment containing the facet) are rational numbers. Hence we may represent the parahoric group $\CG(O_F)$ as the stabilizer group of a point ${\bf a}_\CG$ in the building of $G(K)$ with rational coordinates. We note that if $k$ is algebraically closed, then this first step does not occur.

Now there is a minimal Galois extension $K'$  of  $K$ over which $G$ splits.
  Let $\Theta$ be the Galois group of $K'/K$. Every facet of the building of $G(K)$ consists of  $\Theta$-fixed points contained in a union of finitely many $\Theta$-stable facets of the building of $G(K')$.

From the above, we conclude that the point ${\bf a}_\CG$ in  the building of $G(F)$ has rational coordinates  in terms of any basis of the affine root system of $G/K'$ (determined by a chamber in the apartment containing the vertex of the building of  $G(F)$).

 We claim that the point ${\bf a}_\CG$ is special in the building of $G(F')$, where $F'$ is a suitable   ramified extension of $K'$. (Note that for a split group, special is the same as hyperspecial.)   This statement is also proved in \cite[Lem. 2.4]{Lar}. Let us see what an extension $F'$ of ramification degree $n$ of $K'$  does to an apartment of the building of $G(K')$. Since $G/K'$ is split, any such apartment is also an apartment of the building of $G(F')$, but the facets get subdivided--for example every $1$-dimensional facet gets subdivided into $n$ equal-length facets. Now since, as observed above, the coordinates of  ${\bf a}_\CG$ are rational numbers, by choosing $F'$ of sufficiently divisible ramification degree over $K'$, we can ensure that in the corresponding  building of $G(F')$ every root hyperplane has a parallel root hyperplane that passes through ${\bf a}_\CG$, making ${\bf a}_\CG$ special. 
 
 We now want to find a tamely ramified extension $F'/K$ with the required properties, if $G$ splits over a tame extension of $F$.  Let us first assume that $G$ is absolutely  simple. Then the splitting field $K'$ has degree $1$, $2$, $3$ or $6$ over $K$, and for a classical group degree  at most $2$. We may assume that ${\bf a}_\CG$ lies in the fundamental alcove $\frak a$ inside the standard apartment.  The vertices of $\frak a$ have rational coordinates, where the denominators only involve the prime number $2$ if $G$ is of classical type, resp. $2,3$ if $G$ is not of type $E_8$, resp. $2,3,5$ if $G$ is of type $E_8$.   Hence if ${\bf a}_\CG$ is a vertex, by choosing $F'/K'$ to have ramification degree only containing the above prime factors, we may achieve that ${\bf a}_\CG$ is a special vertex in the building of $G(F')$, comp. also \cite[Lem. 2.2]{Gi}. If $G$ is a classical group, since we excluded $p=2$, the extension $F'/F$ is tame. A similar argument applies if $G$ is not of type $E_8$ (in which case we exclude $p=2,3$), or if   $G$ is of type $E_8$ (in which case we exclude $p=2,3, 5$). If ${\bf a}_\CG$ lies in the interior of a simplex of positive dimension, then, after changing ${\bf a}_\CG$ inside the interior of this simplex, we find a tamely ramified extension $F'$ such that ${\bf a}_\CG$ is a vertex and then proceed as before.  Note that ${\bf a}_\CG$ is fixed under the Galois group of $F'/K$, hence the  parahoric group scheme $\CG'$ corresponding to the vertex ${\bf a}_\CG$ in the building of $G$ over $F'$ is equipped with an action of the Galois group extending the action on $G'=G\otimes_K F'$. 
 
 Now let $G$ be tamely ramified but not necessarily absolutely simple. We can quickly reduce to the case that $G$ is  simple.  Then we may write $G=\Res_{\wt F/F}(H)$, where $\wt F/F$ is a tamely ramified extension and $H$ is an absolutely simple group over $\wt F$ (which splits over a tame extension). Now we apply the argument above to $H/\wt F$ and the parahoric for $H$ corresponding to  ${\bf a}_\CG\in\sB(G, F)=\sB(H, \wt F)$ and deduce the assertion for $G/F$ and the parahoric $\CG$.
\end{proof} 

\begin{proof}[Proof of Theorem \ref{descred}]  Recall that we now assume that $G$ is absolutely simple. Note that under our assumptions on $p$, $G$ splits over a tame extension of $K=k(X)$.
Let $S$ be the finite set of points of $X$ such that $\CG$ is not reductive.   There is a finite extension $k_1$ of $k$ with the following property: For each $x\in S$, the reductive group $G\otimes_K K_{x}$ has the same relative rank over the unramified extension  $K_{x,1}$
of $K_x$ with  residue field $k(x)k_1$, as over the maximal unramified extension of $K_x$ in a separable closure $K_x^{\rm sep}$. Then $G\otimes_KK_{x,1}$ is residually split and quasi-split, for all $x\in S$. We now consider the group scheme $\CG_1=\CG\times_XX_1$ over the base change $X_1=X\otimes_kk_1$
 (which is a curve over $k_1$ in our sense). For simplicity, we denote $X_1$ again by $X$ and omit the subscript $1$ in all that follows. We will apply Proposition \ref{BThyper} and its proof to $\CG\times_X\Spec(O_x)$, $F=K_x$, for $x\in S$: There are tame Galois extensions  $\{F'_x/K_x\mid x\in S\}$ such that $\CG\times_X\Spec(O_x)$ is the stabilizer group scheme of a point ${\bf a}_{\CG, x}$ in the building of $G(K_x)$ and ${\bf a}_{\CG,x}$ is hyperspecial when considered as a point of the building of $G(F'_x)$. In fact, the argument in the proof
shows that the hyperspecial subgroup of $G(F'_x)$ that corresponds to the point ${\bf a}_{\CG, x}$ in the building of $G(F'_x)$
is stable by the action of $\Gal(F'_x/K_x)$. Then the corresponding reductive group scheme $\CG'_{x}$ is equipped with an action of  $\Gal(F'_x/K_x)$  extending the action on $G\otimes_KF'_x$.

We now want to find a tame Galois extension $K'/K$ such that for every place $x'$ of $K'$ over a place $x\in S$, the  extension $K'_{x'}$ of $K_x$ contains a conjugate of $F'_x$. For each $x\in S$, the extension $F'_x/K_x$ is obtained by adjoining the roots of a monic polynomial $P_x$ with coefficients in $K_x$.  By perhaps enlarging these extensions, we may assume that they all have the same degree $d$. Now we find a monic polynomial $P$ of degree $d$ with coefficients in $K$ whose coefficients are sufficiently close to those of $P_x$ for each $x\in S$. Then by Krasner's lemma, the local field extension of $K_x$ generated by the roots of $P$ is isomorphic to $F'_x$, for every $x\in S$. The roots of $P$ generate a field extension $K''/K$ which is separable of degree $d$ and such that $K''_x\simeq F'_x$ for all $x\in S$. Now the normal hull of $K''$ yields the desired extension $K'/K$.  Let $k'$ be the algebraic closure of $k$ in $K'$. Then $k'$ is a finite extension of $k$ and $X'$ is an algebraic curve over $k'$, which is a tame Galois cover of $X\otimes_k k'$ which satisfies our requirements.

Let $\pi\colon X'\to X$ be the cover corresponding to $K'/K$. Then $\CG$ defines a BT group scheme $\CG'$ over $X'$ such that $\CG'_{|X'\setminus\pi^{-1}(S)}=\CG\times_X (X\setminus \pi^{-1}(S))$ and such that for every $x'\in X'$ the group scheme $\CG'\times_{X'}\Spec(O'_{x'})$ corresponds to the hyperspecial point ${\bf a}_{\CG, \pi(x')}$ of the building of $G(K'_{x'})$, cf. the beginning of section \ref{ss:BTcov}  and Remark \ref{RemarkReductive}. Hence all localizations of 
$\CG'$ are reductive, and hence $\CG'$ is a reductive group scheme over $X'$. By  construction, we see that $\CG'$ supports an action 
of the Galois group $\Gamma=\Gal(K'/K)$ which covers the $\Gamma$-action on $X'$ and that $\CG$ and $\CG'$ are $\Gamma$-related. 
\end{proof}

\subsection{Some variants and further discussion}
It is of interest to know when we can achieve the conclusion in Theorem \ref{descred} without first making a finite extension of the base field $k$. 
This can be done provided we exclude a few more ``small characteristics" and is a consequence
of the following variant of the local Proposition \ref{BThyper}:

\begin{proposition}\label{BThyper2}
Let $G$ be an absolutely simple simply connected group over a local field $F$, and let $\CG$ be a parahoric group scheme for $G$ over $O_F$. 
Assume $p$ does not divide $2(n+1)$ if $G$ is of type $A_n$,
 $p\neq 2$ if $G$ is classical  not of type $A_n$, 
 $p\neq 2, 3$ if $G$ is not classical and not of type $E_8$, and $p\neq 2, 3, 5$ if $G$ is of type $E_8$. 
 
 Then there exists a finite tame Galois  extension $F'/F$ such that $G\otimes_FF'$ is split and a hyperspecial group scheme $\CG'$ of $G(F')$ such that  $\CG'$ is stable under the action of $\Gal(F'/F)$ on $G\otimes_FF'$ and such that 
\[
\CG=({\rm Res}_{O_{F'}/O_F}\CG')^{{\rm Gal}(F'/F)}.
\] 
\end{proposition}

\noindent(This identity implies $\CG(O_F)=G(F)\cap\CG'(O_{F'})$, $\CG(O_{\breve F})=G({\breve  F})\cap \CG'(O_{{\breve F}'})$.)

\begin{proof}
We re-examine the proof of Proposition \ref{BThyper}, keeping the same notation. Here, we assume that $G$ is absolutely simple. We will show that 
there is a $\CG'$ as in the conclusion of Proposition \ref{BThyper} which, in addition, is $\Gal(F'/F)$ stable. The argument shows that the desired statement will follow if there is a tame extension $F'/F$, constructed as in the proof of Proposition \ref{BThyper}, which is Galois and is such that the (hyperspecial) vertex ${\bf a}_{\CG}\in \sB(G, F')$ which appears in the last 
step of the argument is fixed by the action of the Galois group of $F'/F$ (and not just of $F'/K$).
This can be arranged, if we can express $\CG$ as the stabilizer of a point ${\bf a}_\CG\in \sB(G,F)$ whose coordinates in  $\sB(G, K')$ have denominators 
which are relatively prime to the characteristic of $p$. Recall that, since $K/F$ is an unramified Galois extension,  the vertices in $\sB(G,F)$ are the barycenters of minimal 
$\Gal(K/F)$-stable facets in $\sB(G, K)$. The action of $\Gal(K/F)$ on an alcove of $\sB(G,K)$ is described via automorphisms of the corresponding 
(possibly twisted) affine Dynkin diagram. The coordinates of these barycenters are sums of the coordinates of vertices divided by numbers which are factors of the order 
of the image of  $\Gal(K/F)$ in the automorphism group of the  affine Dynkin diagram. Note that the automorphism group
of a  affine Dynkin diagram is either trivial, or has order which involves only the prime factors $2$ and $3$ in all cases, except 
in the case of (untwisted) type $A_n$, in which the order is  $2(n+1)$ (see for example the tables in \cite[\S 6] {Gross}, cf. Remark \ref{RemarkAn} below).
By using this, combined with the arguments in the proof of Proposition \ref{BThyper}, 
we obtain: 

The group scheme $\CG$ is the stabilizer of a point ${\bf a}_\CG\in \sB(G,F)$ whose coordinates in $\sB(G,K')$ have denominators which are $1$ or only involve the following prime numbers:
\begin{itemize}
\item $2$, for all classical types except $A_{n}$, 

\item $2$ and  the prime divisors of  $n+1$, for type $A_{n}$,

\item $2$, $3$, for all types except $A_{n}$ and $E_8$,

\item $2$, $3$, $5$, for type $E_8$.
\end{itemize} 
Using also the discussion in \S \ref{ss:BTcov}, this now implies the statement.
\end{proof}
\quash{implies that, for semi-simple simply connected tamely ramified classical groups $G$ without $A_{n}$ factors and for $p\neq 2$, the  alternative conclusion $(*)$ to Proposition \ref{BThyper}, in a) above, holds.\mar{careful, the general not absolutely simple case might require more care, also in the proof of the Theorem...}
 More generally, if $G$ is classical and tamely ramified, it holds when the characteristic $p$ does not divide $2(n+1)$, for all $n$ such that $G$ has a factor of type $A_n$. For non-classical groups (including trialitarian forms) the list of excluded characteristics for the stronger statement
remains the same as for the original statement.}

 \begin{remark}\label{RemarkAn}
a) Suppose $G$ is of type $A_n$ and let $\sR$ be the corresponding (twisted) affine Dynkin diagram. There are two possibilities
$\sR=\sA_n$, or $\sR={}^2\sA_n$. The automorphism groups are ${\rm Aut}(\sA_n)\simeq \BZ/(n+1)\BZ\rtimes S_2$, or ${\rm Aut}({}^2\sA_{2m+1})\simeq   S_2$ or ${\rm Aut}({}^2\sA_{2m})=(1)$ (\cite[\S 6]{Gross}).
The Galois group $\Gal(K/F)$ acts on the base alcove of $\sB(G, K)$ via a homomorphism 
\[
 \Gal(K/F)\to {\rm Aut}(\sR).
 \]
 Denote by $\delta$ the order of the image of this homomorphism; this order is a divisor of $2(n+1)$. The argument in the proof above actually shows that the conclusion of Proposition \ref{BThyper2} is true for $G$, provided we just 
 exclude the primes $p$ that divide $2\delta$. 
 
b) Suppose that $F$ has finite residue field. By consulting the list in \cite[\S 7]{Gross}, we see that $\delta >2$ only when 
 $G=A^\times_1$, the norm one units of a $F$-central simple algebra $A\simeq {\rm M}_{r\times r}(D)$, with $D$ a division algebra
 of index $m>1$. In this case, $n+1=rm$ and $\delta=m$.

c) Consider $G=D_1^\times$, the norm one units
 of a $F$-central division  algebra $D$ of invariant $1/d$, where $F$ has finite residue field. In this case, the affine Dynkin diagram is $\sA_d$ and $\delta=d$.
Take $\CG$ to be the unique Iwahori of $G$. Here the building $\sB(G, F)=\{*\}$ is a one point set and ${\bf a}_{\CG}=*$. In the proof of Proposition \ref{BThyper2}, we can take $K$ to be the unramified extension of $F$ of degree $d$. The point $*$ is the barycenter of an alcove in $\sB(G, K)=\sB({\rm SL}_d, K)$ and has coordinates with denominator $d$. Unless we base change to an extension of $F$ with ramification degree divisible by $d$, we cannot make $*$ hyperspecial. Using this, we can see that for this example, the conclusion of Proposition \ref{BThyper2}  does not hold if the characteristic $p$ divides $d=\delta$. 
 \end{remark}
 
 We can now deduce the following variant of the global Theorem \ref{descred}. We keep the same notations.
 
 \begin{theorem}\label{descred2}
  Assume that $G$ is absolutely simple and simply connected.
  Also assume $p$ does not divide $2(n+1)$ if $G$ is of type $A_n$, $p\neq 2$ if $G$ is classical not of type $A_n$, 
 $p\neq 2, 3$ if $G$ is not classical and not of type $E_8$, and $p\neq 2, 3, 5$ if $G$ is of type $E_8$. 
 
 Let $\CG$ be a Bruhat-Tits group scheme for $G$ over $X$. There exists a tame Galois covering $X'/X$ with Galois group $\Gamma$ and a reductive group scheme $\CG'$ over $X'$  such that $\CG$ and $\CG'$ are $\Gamma$-related. 
\end{theorem}
 
 \begin{proof}
 Follows by the argument in the proof of Theorem \ref{descred} by using Proposition \ref{BThyper2} instead of Proposition \ref{BThyper}.
 Note that in this case, we do not need the finite base field extension $k_1$ of $k$ which appears in the proof of Theorem \ref{descred}. On the other hand, we have to accept that the covering $X'$ is a curve over a finite extension $k'$ of $k$.
 \end{proof}
 
  \begin{remark}\label{RemarkAn2}
 Suppose that $G=A^\times_1$ is the group of   norm one units of  $A={\rm M}_{r\times r}(D)$, with $D$ a division $k(X)$-algebra 
of index $m=(n+1)/r$. Then Theorem \ref{descred2} holds for $G$ after replacing $2(n+1)$ by $2m$.
\end{remark}

 \begin{remark}\label{Gquas}
We therefore see that, excluding small characteristics, parahoric group schemes on the given curve $X$ for an absolutely simple simply connected group arise from reductive group schemes over coverings $X'/X$ by the simple operation of ``$\Gamma$-invariants in restrictions of scalars". Such group schemes are considered in characteristic zero by Balaji-Seshadri \cite{BalaS} and Damiolini \cite{Dam} and Hong-Kumar \cite{HoKu} when the reductive group scheme over $X'$ is constant, i.e., of the form $H\times_{\Spec(k)} X'$ for $H$ over $k$. However, in general, reductive group schemes over $X'$ are not of this form. An example is given by $\CG'=\GL(\CV)$ for a vector bundle on $X'$. 
 Another example is given by  the unitary group associated to an \'etale double covering $\wt X'/X'$ and a nowhere degenerate hermitian form on a vector bundle on $\wt X'$, provided $p\neq 2$. 
 
   It seems impossible to classify all reductive groups over $X'$. Indeed, for vector bundles $\CV$ and $\CV'$, we have $\GL(\CV)\simeq\GL(\CV')$ iff $\CV\simeq\CV'\otimes\CL$ for a line bundle $\CL$, hence classifying reductive groups over $X'$ with generic fiber $\GL_n$ is essentially equivalent to classifying all vector bundles of rank $n$, which is impossible if $n\geq 2$ and $X'$ has higher genus.   
\end{remark}

 Note that when $k$ is algebraically closed, then $G'$ is quasi-split (Steinberg's theorem \cite[Ch. III, \S 2.3, Thm. $1'$, plus Rem. 1 at end of \S 2.3]{Serre}). The following proposition deals with the quasi-splitness over $X'$ (existence of a \emph{relative} Borel subgroup).
\begin{proposition}\label{Hasse}
Let $\CG$ be a reductive group scheme over a curve $X$. Assume that the base field $k$ is either algebraically closed or finite. Then $\CG$ admits a Borel group over $X$, i.e., there exists a smooth closed subgroup scheme $\CB$ of $\CG$ which is a Borel subgroup in every fiber of $\CG$. 
\end{proposition}
\begin{proof}
Let first $k$ be algebraically closed. Then the generic fiber of $\CG$ is quasi-split, as mentioned above. Therefore there exists a Borel subgroup in the generic fiber. This Borel extends over all of $X$ by the valuative criterion of properness, since the scheme of Borel subgroups is projective over $X$. 

Now let $k$ be finite. The previous argument shows that it suffices to prove that the generic fiber $G$ of $\CG$ is quasi-split. But every localization $\CG_x$ is quasi-split. Indeed, this follows by the smoothness of the scheme of Borel subgroups from the fact that the reduction $\bar\CG_x$ is quasi-split (any linear algebraic group over a finite field is quasi-split, cf. \cite[Ch. III, \S 2.2, Thm. 1 plus Ch. III, \S 2.3, Ex. a) for Thm. $1'$]{Serre}). Let $G_0$ be the quasi-split inner form of $G$. The assertion now follows from the exactness of the localization sequence (\emph{Hasse principle for adjoint groups over a global field}),
\[
0\to \RH^1(K, G_{0,\ad})\to \prod\nolimits_{x\in X} \RH^1(K_x, G_{0,\ad}) .
\]
{We refer to the appendix for a proof of the Hasse principle for adjoint groups over a global field of positive characteristic. }
\end{proof}
The following proposition clarifies the nature of the reductive group schemes over $X$ appearing in \cite{BalaS}. 
\begin{proposition}\label{Gform}
Let $\CG$ be a reductive group scheme over the curve $X$ such that the generic fiber $G$ is constant, i.e., of the form $G=H\otimes_k K$. Then $\CG$ is a form of $H$, i.e., locally on $X$ (for the \'etale or fppf topology), $\CG$ is isomorphic to the constant group scheme $H\times_{\Spec(k)} X$. Furthermore, there is an open subset $U\subset X$ such that $\CG_{|U}$ is constant, i.e., isomorphic to $H\times_{\Spec k} U$. 
\end{proposition}

\begin{proof} The first part of the statement easily follows from the fact that reductive group schemes are, locally for the \'etale topology on the base, split; see for example \cite[Lemma 5.1.3]{Conrad}. The second
claim quickly follows from the fact that any isomorphism between $\CG\times_X \Spec(K)$ and $H\otimes_k K$ extends to an open subset of $X$. 
\end{proof}

\begin{example}\label{BSHex} The following class of examples of stabilizer BT group schemes over $X$ is considered by Balaji-Seshadri \cite{BalaS} and Hong-Kumar \cite{HoKu}. Assume for simplicity that $k$ is algebraically closed. Let $X'/X$ be a tame Galois covering with Galois group $\Gamma$. Let $G$ be a reductive group over $K$.  We assume  that the base change group $G'$ is a constant group scheme, i.e., $G'=H\times_{\Spec(k)}\Spec(K')$ for a reductive group $H$ over $k$ (which is split since $k$ is algebraically closed). Then $G$ arises by descent and corresponds to a cohomology class $c\in \RH^1(\Gamma, \Aut(H\otimes_k K'))$.   Let $\CG'=H\times_{\Spec(k)} X'$, the constant reductive group scheme over $X'$.  Let $\CG$ be defined by descent from $\CG'$ by a cocycle
$\theta\colon \Gamma\to  \Aut(H\times_k X')=\Aut(H)$ representing the given cohomology class $c$. Then  $\CG'$ is a stabilizer BT group scheme, cf. Remark \ref{stabBT}, iii), and hence so is $\CG$, and  $(\CG, \CG')$ are $\Gamma$-related. Since the Galois group $\Gamma$ acts trivially on $\Aut(H)$, the cocycle condition says that  $\theta \colon \Gamma\to\Aut(H)$  is a group homomorphism. 

After choosing a pinning, we have an isomorphism $\Aut(H)\simeq H_\ad(k)\ltimes \Aut({\rm Dyn}(H))$. We  consider the following two extreme cases.

\smallskip

1)  $\theta\colon \Gamma\to H_\ad(k)$. 
In fact, for simplicity, we assume $\theta$ lifts to a homomorphism $\tilde\theta\colon \Gamma\to H(k)$. 
In this case, by Steinberg's theorem, the  cohomology class $\tilde c\in \RH^1(\Gamma, H(K'))$ obtained by composing $\tilde\theta$ with $H(k)\to H(K')$ is trivial, 
i.e., there exists $ g\in H(K')$ such that $\tilde\theta(\gamma)=   g^{-1}\gamma( g)$; hence $G\simeq H\times_{\Spec(k)}\Spec(K)$ is constant. Then, under this identification, we have $\CG(O_x)\simeq H(K_x)\cap g H(O'_{x'})g^{-1}$ for every $x'\in X'$.
In other words, $\CG$ corresponds to the stabilizer BT group scheme corresponding to the collection of points  $\{{\bf a}_x=g\cdot {\bf a}_{0, x}\in \sB^e(H, K_x)\mid x\in X\}$, where ${\bf a}_{0, x}$ denotes  the base point of the building of the split group $H$ over $K'_{x'}$. Note that the point ${\bf a}_x$ is independent of the choice of $x'$ over $x$ and lies in the building over $K_x$.

\smallskip

2) $\rho\colon \Gamma\to \Aut({\rm Dyn}(H))$. In this case, $G$ is a quasi-split outer form of $H$ over $K$ and, for every point $x\in X$, the localization $\CG_x$ is the stabilizer BT group scheme corresponding to a point ${\bf a}_x\in\sB^e(G, K_x)$ which, when considered as a point of the building $\sB^e(G, K'_{x'})=\sB^e(H, K'_{x'})$, is a special point of the split group $H$ over $K'_{x'}$.
\end{example}

\section{Tamely ramified $\CG$-bundles and $(\CG',\Gamma)$-bundles}\label{s:tamebdls}
 
Recall that by the Tannakian formalism of Broshi \cite{Bro}, for any smooth affine group scheme $\CG$ over a Dedekind scheme $X$, the two possible notions of $\CG$-bundles coincide:   $\CG$-torsors (for the fpqc-topology, or the fppf-topology, or the \'etale-topology, these are all equivalent), and  fiber functors on the category of representations of $\CG$ (i.e., group scheme homomorphisms $\rho\colon\CG\to \GL(\CV)$, where $\CV$ is a vector bundle on $X$). 
\subsection{$(\CG', \Gamma)$-bundles}\label{ggamma}
The following key definition is due to Balaji-Seshadri \cite{BalaS}.
\begin{definition}\label{GGammadef}
 Let $X'/X$ be a finite  Galois covering with Galois group $\Gamma$.  Let $\CG'$ be a   stabilizer Bruhat-Tits group scheme over $X$ equipped with a lifting of the action of $\Gamma$.   A $(\CG', \Gamma)$-bundle is a $\CG'$-bundle with a $\Gamma$-action compatible with the $\Gamma$-action on $\CG'$. 
\end{definition}
There is the following stack theoretic interpretation of $(\CG', \Gamma)$-bundles. 
Consider the quotient stack $\fkX:=[X'/\Gamma]$ which supports a natural  morphism given by taking  quotients by the action of $\Gamma$,
\begin{equation}\label{stacktoX}
q: \fkX=[X'/\Gamma]\to X .
\end{equation}
  The group scheme $\CG'$ over $X'$ with its semi-linear $\Gamma$-action gives $\fkG:=[\CG'/\Gamma]$. The natural morphism
$\fkG:=[\CG'/\Gamma]\to \fkX=[X'/\Gamma]$ is representable by a (relative) smooth affine group scheme. The groupoid of $(\CG',\Gamma)$-bundles over $X'$ 
is naturally equivalent to the groupoid of $\fkG$-bundles (i.e. $\fkG$-torsors) over $\fkX$.
This can be seen by using the Tannakian description and a corresponding equivalence statement for the exact tensor categories of vector bundles.

\begin{proposition}\label{BTtors}
Let $X'/X$ be a finite tame Galois covering with Galois group $\Gamma$. Let $\CG$ and $\CG'$ be $\Gamma$-related  stabilizer Bruhat-Tits group schemes over $X$, resp. $X'$. Any $\CG$-bundle $\CP$ on $X$ defines a $(\CG', \Gamma)$-bundle $\CP'$ on $X'$ such that $\CP=\Res_{X'/X}(\CP')^\Gamma$. This defines a fully faithful functor
$$
\text{$\{ \CG$-bundles on $X\}\mapsto \{ (\CG', \Gamma)$-bundles on $X'\}$ .}
$$

\end{proposition}
\begin{proof}
Let $\CP$ be a $\CG$-bundle. The associated $(\CG', \Gamma)$-bundle $\CP'$ is characterized by two properties. Recall the ramification locus $R\subset X'$ of $\pi: X'\to X$. First, the restriction of $\CP'$ to $X'\setminus R$ is equal to the pullback of $\CP$ to $X'\setminus R$. Second, for each $x'\in R$, there is a  map of torsors
\[
\CP\times_{X}\Spec(O'_{x'})\to \CP'\times_{X'}\Spec(O'_{x'}),
\]
compatible with the natural map of group schemes $\CG\times_{X}\Spec(O'_{x'})\to \CG'\times_{X'}\Spec(O'_{x'})$.

Let $\CP_1$ and $ \CP_2$ be two $\CG$-bundles with associated $(\CG', \Gamma)$-bundles $\CP_1'$ and $\CP_2'$.  We need to see that any isomorphism $\varphi':\CP_1'\to\CP_2'$ comes from a unique isomorphism $\varphi: \CP_1\to\CP_2$. But $\CG=\Res_{X'/X}(\CG')^\Gamma$ implies that the natural map $\CP_i\to\Res_{X'/X}(\CP_i')^\Gamma$ is an isomorphism for $i=1, 2$, whence the assertion. 
\end{proof}

Let $\CP'$ be a $(\CG', \Gamma)$-bundle on $X'$. Then $\Res_{X'/X}(\CP')^\Gamma$ is equipped with a $\CG$-action. The observation of Damiolini \cite{Dam} is that it is not always true\footnote{The example of Damiolini does not involve parahoric group schemes, and therefore one might argue that it possibly does not apply in our context. However, Proposition \ref{propLeray} shows that it does.} that $\Res_{X'/X}(\CP')^\Gamma$ is a $\CG$-bundle on $X$. In other words, the functor in Proposition \ref{BTtors} is not essentially surjective. 

\begin{remark}
Let $\CG$ be a reductive group scheme over $X$ such that its generic fiber is constant. By Proposition \ref{Gform}, $\CG$ is a form of $H$. Assume that $\CG$ is a strong inner form, i.e., $\CG=\underline{\Aut}(\CP_0)$, for a $H$-torsor $\CP_0$ on $X$.  Then there is an equivalence of categories between the category of $\CG$-torsors on $X$ and  the category of $H$-torsors on $X$: indeed, $\CP_0$ defines a fiber functor with values in the category of vector bundles on $X$ of the category of representations of $H$. In other words, $\CP_0$ is a right-$H$-torsor and a left-$\CG$-torsor, and the equivalence is given by $\CP\mapsto \CP\times^{\CG}\CP_0$. This remark applies in particular to $H=\GL_n$, and shows that for any strong inner form $\CG$ of $\GL_n$, the category of $\CG$-bundles on $X$ is equivalent to the category of vector bundles of rank $n$ on $X$. For example, if $\CG=\GL(\CV)$, for a vector bundle $\CV$ of rank $r$ on $X$, then the category of $\CG$-bundles is equivalent to the category of vector bundles of rank $r$ on $X$.
\end{remark}

\section{Purity for $(\CG',\Gamma)$-bundles}\label{s:purity}

In this section, we continue to assume that $X'/X$ is a finite tame Galois covering of curves over $k$ with Galois group $\Gamma$. Let $\CG$ and $\CG'$ be $\Gamma$-related  stabilizer Bruhat-Tits group schemes over $X$, resp. $X'$. 
 
 Recall the notion of an $S$-family of $\CG$-bundles on $X$: by definition, this is a $\CG$-bundle on $S\times X$. Similarly, there is the notion of  an $S$-family of a $(\CG', \Gamma)$-bundles on $X'$:    by definition, this is a $\CG'$-bundle on $S\times X'$ equipped with a semi-linear $\Gamma$-action (semi-linear for the $\CO_{X'}$-factor). Let $\CP'$ be an $S$-family of $\CG'$-bundles  on $X'$ with semi-linear action of $\Gamma$. Then $\Res_{X'/X}(\CP')^\Gamma$  is a smooth scheme over $S\times X$, equipped with an action of the parahoric group scheme $\CG$, comp. Proposition \ref{BTtors}. In fact, it is a \emph{pseudo-torsor} under $\CG$, i.e., if $\Res_{X'/X}(\CP')^\Gamma(T)\neq\emptyset$, for a $S$-scheme $T$, then the action of $\CG(T\times X)$ on $\Res_{X'/X}(\CP')^\Gamma(T)$ is simply transitive. 

\begin{proposition}
There is a maximal open subset $U\subset S$ such that $(\Res_{X'/X}(\CP')^\Gamma)_{|U\times X}$ is a $\CG$-torsor. 
\end{proposition}

\begin{proof}
What is preventing $\Res_{X'/X}(\CP')^\Gamma$ from being a $\CG$-torsor is that the fiber over $(s, x)\in S\times X$ may be  empty. The set of such $(s, x)$ is closed by the smoothness of $\Res_{X'/X}(\CP')^\Gamma$. Therefore, $U$ is the complement of the projection of this closed subset to $S$ (which is closed by the properness of $X$). 
\end{proof}

The purity theorem in question is the following statement.

\begin{theorem}\label{thmpur}
Let $\CP'$ be  an $S$-family of $(\CG', \Gamma)$-bundles on $X'$. Assume that there is an open dense subset $U\subset S$ such that the induced $U$-family $(\Res_{X'/X}(\CP')^\Gamma)_{|U\times X}$ is a $U$-family of $\CG$-bundles. Then $\Res_{X'/X}(\CP')^\Gamma$ is an $S$-family of $\CG$-bundles. 
\end{theorem}

\begin{proof}
  Let $B$ be the branch locus of $X'\to X$, in other words, denoting by $R'$ the inverse image of $B$, the morphism $X'\setminus R'\to X\setminus B$ is \'etale, and $B$ is minimal with this property. It is easy to see that the restriction of $\Res_{X'/X}(\CP')^\Gamma$ to $S\times (X\setminus B)$ is a $\CG$-bundle. Let $s\in S\setminus U$ and $x\in B$. Then what has to be shown is that the fiber of $\Res_{X'/X}(\CP')^\Gamma$ in $(s, x)\in S\times X$ is non-empty. Indeed, if this fiber is non-empty, a section can be lifted to a section of ${\rm Res}_{X'/X}(\CP')^\Gamma$ over an open neighbourhood  (use the smoothness of ${\rm Res}_{X'/X}(\CP')^\Gamma$), hence by varying $x$, we see that  ${\rm Res}_{X'/X}(\CP')^\Gamma$ is a $\CG$-torsor over $S\times X$ (use the properness of $X$). 

Assume, by way of contradiction, that the fiber of ${\rm Res}_{X'/X}(\CP')^\Gamma$ over $(s, x)$ is empty. By making a base change $\Spec(V)\to S$, where $V$ is a complete DVR with residue field $k(s)$, mapping the special point to $s$ and the generic point to $U$, we may assume that $S=\Spec(V)$. Also, by localizing around $(s, x)$ and completing, we can consider the situation over $S\times X_x$, where $X_x$ is the completion of $X$ at $x$. Let $\CV=(S\times X_x)\setminus (s, x)$. Then $\CV$  is the punctured spectrum of a 2-dimensional regular local ring $R$. The restriction of ${\rm Res}_{X'/X}(\CP')^\Gamma$ to $\CV$ is a $\CG$-torsor. If the fiber at the special point is empty, then $\CP:=({\rm Res}_{X'/X}(\CP')^\Gamma)_{|\CV}={\rm Res}_{X'/X}(\CP')^\Gamma$ is an affine scheme. Let $X'_x=X_x\times_X X'$, and set $\CV'=\CV\times_{X_x} X'_x$. Then we see that  $\CP\times_{\CV}\CV'=\CP\times_{X_x} X_x'$ is also affine. Consider the push-out morphism
$$
p\colon \CP\times_{\CV}\CV'\to \CP'_{|\CV'}. 
$$
It induces a map on cohomology,
\begin{equation}\label{pusho}
{\rm H}^1(\CP'_{|\CV'}, \CO)\to {\rm H}^1(\CP\times_\CV \CV', \CO) .
\end{equation}
The map \eqref{pusho} is an isomorphism up to bounded $\pi$-torsion, where $\pi$ denotes a uniformizer of $X_x$. Indeed, since $p$ is an affine morphism, the map  \eqref{pusho}  is induced by the map of sheaves on $\CP'_{|\CV'}$ given by 
$$
\CO_{\CP'_{|\CV'}}\to p_*(\CO_{\CP\times_\CV \CV'}) ,
$$
and this map is injective, with cokernel a skyscraper sheaf on $\CV'\times_{\CV} \big((S\times \{x\})\cap\CV\big)$. 

Now the target of the map \eqref{pusho} is ${\RH}^1(\CP\times_\CV \CV', \CO)=0$ since $\CP\times_\CV \CV'$ is affine. 
On the other hand,  $\CP'_{|\CV'}$ is a trivial $\CG'$-torsor over $\Spec (R')=S\times X'_x$, which is a product of complete local rings. Hence the source of the map \eqref{pusho} can be identified with ${\rm H}^1(\CV', \CO)\otimes {\rm H}^0(\CG', \CO)$. Let $\mathfrak m$ be the maximal ideal of $R'$. Since ${\rm H}^1(\CV', \CO)={\rm H}^2_{\mathfrak m}(R')$ is not of bounded $\pi$-torsion (it contains the images of $\frac{1}{\pi^a\varpi^b}\in R'_{\pi\varpi}$ for any $a>0, b>0$, where $\varpi$ is a uniformizer of $V$), this is the desired contradiction. 
\end{proof}

\begin{remark}
The argument in the proof of Theorem \ref{thmpur} is due to P.~Scholze and is used in \cite[\S 5]{PRglsv} to give a simple proof of Ansch\"utz' extension theorem in the case of \emph{essentially tamely ramified} groups.
\end{remark}

\section{Types of local $(\CG', \Gamma)$-bundles} \label{s:localt}

In this section, we consider the local analogue of $(\CG', \Gamma)$-bundles. More precisely, let  $O'/O$ be a finite extension
of \emph{strictly henselian} dvr's with Galois group $\Gamma={\rm Gal}(F'/F)$. Let $G$ be a reductive group over $F$, and $G'$ its base change to $F'$. Also, let $\CG'$ be a smooth group scheme over   $\Spec(O')$ with  generic fiber $G'$ which supports a compatible  $\Gamma$-action.

The notion of a  $(\CG',\Gamma)$-bundle over $\Spec(O')$ is clear, cf. Definition \ref{GGammadef}. As usual, we denote by $\RH^1(\Spec(O');\Gamma, \CG')$ the set of isomorphism classes of these bundles, cf. \cite[Chap. V]{Toho}. 
We call this set, the ``set of \emph{(local) types} of $(\CG',\Gamma)$-bundles over $O'$".

\subsection{Types}\label{ss:types}
The following proposition gives a classification of $(\CG',\Gamma)$-bundles over $O'$.
\begin{proposition}\label{LTprop}
There are natural bijections 
\[
\RH^1(\Spec(O');\Gamma, \CG')\xrightarrow{\ \sim\ } \RH^1(\Gamma, \CG'(O'))  \xrightarrow{\ \sim\ } G(F)\backslash \big(G'(F')/\CG'(O')\big)^\Gamma.
\]
\end{proposition}

\begin{proof} The existence of the first bijection follows easily from the definitions. Indeed, since $\CG'$ is smooth and $O'$ strictly henselian, every $(\CG',\Gamma)$-bundle $\CP'$ over $\Spec(O')$ is trivial as a $\CG'$-bundle; after picking a section of $\CP'$ the structure as $(\CG',\Gamma)$-bundle is described by a $1$-cocycle $\theta: \Gamma\to \CG'(O')$. 

It remains to explain the second bijection. 
 The essential ingredient here is Steinberg's theorem which implies that $\RH^1(\Gamma, G'(F'))=0$. 
 Using this, the result quickly follows from  \cite[I \S 5, Prop. 36, Cor. 1]{Serre}. In fact, the bijection 
 \[
\RH^1(\Gamma, \CG'(O'))  \xrightarrow{\ \sim\ } G(F)\backslash \big(G'(F')/\CG'(O')\big)^\Gamma.
\]
is given as follows. Starting  with  a $1$-cocycle  $\theta: \Gamma\to \CG'(O')$, by Steinberg's theorem there is $h\in G'(F')$ such that  $\theta(\gamma)=h^{-1}\gamma(h)$. Since 
$\theta(\gamma)\in \CG'(O')$ we have $h\cdot \CG'(O')\in (G'(F')/\CG'(O'))^\Gamma$ and the map is given by sending the class of the cocycle $\theta$ to the $G(F)$-orbit of $h\cdot \CG'(O')$. 
\end{proof}

\begin{remark}\label{remfinitevar}
Suppose that $O'$ is a complete dvr with finite residue field, $G'$ is semi-simple and simply connected, and $\CG'$ has connected special fiber. Then, the conclusion of Proposition \ref{LTprop} continues to hold. The proof is  by the same argument using the following ingredients: By smoothness of $\CG'$ and Lang's theorem on the triviality of $\RH^1$ of a connected algebraic group over a finite field (see \cite[Ch. III, \S 2.3, Ex. a) for Thm. $1'$]{Serre}), we see that every $(\CG',\Gamma)$-bundle $\CP'$ over $\Spec(O')$ is trivial as a $\CG'$-bundle, and we again have $\RH^1(\Gamma, G'(F'))=0$ by work of Bruhat-Tits  (\cite[Thm. 4.7]{BTIII}). 
\end{remark}

\subsection{Types for BT group schemes}\label{typesforBT} Assume now that $F'/F$ is  a tamely ramified extension. Then the extended building $\sB^e(G, F)$ is identified with the $\Gamma$-fixed points $\sB^e(G', F')^\Gamma$ of the extended  building $\sB^e(G', F')$ under its Galois $\Gamma$-action \cite{PraYu}. Assume also that $\CG'$ is a BT group scheme and that, in fact,  $\CG'( O)$ is the stabilizer $G'(F)_{\bf a}$ in $G'(F)$ of a point ${\bf a}\in \sB^e(G', F')$ which is fixed by $\Gamma$, i.e.  ${\bf a}\in \sB^e(G, F)$. Let $\CG=({\rm Res}_{O'/O}\CG')^\Gamma$. Then $\CG$ is a BT group scheme over $O$ and $\CG( O)$ is the stabilizer $G(F)_{\bf a}$ in $G(F)$ of the  point $\bf a$ considered now in the building $\sB^e(G, F)$.  In other words, $\CG$ and $\CG'$ are $\Gamma$-related stabilizer BT group schemes. 

\begin{lemma}\label{LTlemma} Under the above assumptions, there is a bijective map
\[
(G'(F')/\CG'(O'))^\Gamma\xrightarrow{\ \sim\ } \sB^e(G, F)\cap \big(G'(F')\cdot \ome\big)
\]
given by $
h\cdot \CG'(O')\mapsto h\cdot \ome$. In the target, the intersection takes place in $\sB^e(G', F')$. 
\end{lemma}

\begin{proof} Follows from the definitions, since $\CG'(O')$ is the stabilizer of $\bf a$ in $G'(F')$
and since we have $\sB^e(G, F)=\sB^e(G', F')^\Gamma$.
\end{proof}
Combining with Proposition \ref{LTprop} this gives

\begin{proposition}\label{LTprop2}
There are natural bijections 
\[
\RH^1(\Spec(O');\Gamma, \CG')\xrightarrow{\ \sim\ } \RH^1(\Gamma, \CG'(O'))  \xrightarrow{\ \sim\ } G(F)\backslash \big(\sB^e(G, F)\cap (G'(F')\cdot {\ome})\big).
\]\qed
\end{proposition}

\begin{corollary}
The set of types $\RH^1(\Spec(O');\Gamma, \CG')$ is finite.
\end{corollary}

\begin{proof} Every orbit of the $G(F)$-action on $\sB^e(G, F)$ has a representative in  the closure of an alcove in $\sB^e(G, F)$, and such a closed alcove has a finite intersection with a $G'(F')$-orbit.  \end{proof}

\begin{example}\label{eqcar}
Assume that $O= k[[\varpi]]$ and $O'= k[[\varpi']]$, where $ k$ is an algebraically closed field. 
Assume that $\CG$ and $\CG'$ are  $\Gamma$-related stabilizer BT group schemes as above. Then $G'(F')/\CG'(O')$ is in equivariant bijection with the set of $ k$-valued points of the (partial) affine flag variety 
 ${\sF}_{\CG'}=LG'/L^+\CG',$
  cf. \cite{PRtwisted}. Hence, $(G'(F')/\CG'(O'))^\Gamma\simeq {\sF}_{\CG'}(k)^\Gamma$ are the $\Gamma$-fixed points of ${\sF}_{\CG'}( k)$; the group $G(F)=LG(k)$ acts  and   Proposition \ref{LTprop}
gives
\[
\RH^1(\Spec(O');\Gamma, \CG')\xrightarrow{\ \sim\ } \RH^1(\Gamma, \CG'(O'))\xrightarrow{\ \sim\ }  LG( k)\backslash {\sF}_{\CG'}( k)^\Gamma.
\]

\end{example}
\begin{remark}\label{redofLT}
Denote by $\bar \CG'^{\rm red}$ the maximal ``reductive" quotient of the special fiber $\bar \CG'$ of $\CG'$ over $O'$, i.e., the quotient of $\bar \CG'$ by the unipotent radical of its neutral component. Note that as an exemption to our usual conventions, this group might not be connected (whence the quotation marks). However, connectedness holds  if $\CG'$ is a parahoric group scheme.  

The kernel of the natural homomorphism $\CG'(O')\to \bar\CG'^{\rm red}( k)$ 
is the $k$-points of a pro-unipotent group  and this map is surjective by smoothness of $\CG'$.  Since,  by our tameness hypothesis, the order of $ \Gamma$ is prime to the residue characteristic of $k$, a standard argument gives that $\CG'(O')\to \bar\CG'^{\rm red}(k)$  induces a bijection
\begin{equation}\label{redofcoho}
\RH^1(\Gamma, \CG'(O'))\xrightarrow{\ \sim\ } \RH^1(\Gamma, \bar\CG'^{\rm red}( k)).
\end{equation} 
This bijection gives an alternative approach to the explicit determination of the set of local types of $(\CG', \Gamma)$-bundles over $O'$, especially when $\bar \CG'$ is connected. Let us explain this.

Let us assume that $H= \bar\CG'^{\rm red}$ is a connected reductive group over the algebraically closed field $k$. Since the extension $F'/F$ is supposed to be tame, we may identify $\Gamma$ with $\mu_e$, where $e=[F':F]$ is prime to ${\rm char}\, k$. 

By \cite[Thm. 7.5]{Stein}, there exists a maximal torus  $T$ of $H$ and a Borel subgroup $B$ containing $T$ which are  $\Gamma$-stable. Fix such $T$, with Weyl group $W$. Then the natural map $\RH^1(\Gamma, T(k))\to \RH^1(\Gamma, H(k))$ is surjective and induces a bijection
\begin{equation}\label{H1Stein}
\RH^1(\Gamma, H(k))=\RH^1(\Gamma, T(k))/W^\Gamma,
\end{equation}
cf. \cite[Proof of Prop. 2.4]{PZ}. This action is given as follows: Choose a generator $\gamma_0\in\Gamma$. Let $w\in W^\Gamma$ and lift $w$ to $\dot w\in N(T)$. Set $t_w:=\dot w^{-1}\gamma_0(\dot w)$.
Suppose an element $[t]\in \RH^1(\Gamma, T(k))$ is represented by $t\in T(k)$. Then  set $[t]\cdot w=[\dot w^{-1}t\dot w\cdot t_w]=[\dot w^{-1}t\gamma_0(\dot w)]$, i.e. $w$ acts on $t$ by $\gamma_0$-conjugation by $\dot w$. 

 Note that, since $\Gamma$ is a cyclic finite group, there is an identification
\begin{equation}\label{H1torus}
\RH^1(\Gamma, T(k))=\ker \big({\rm N}_\Gamma\colon T(k)\to T(k)\big)/I_\Gamma T(k) ,
\end{equation}
 cf. \cite[\S 8]{AW}. Here ${\rm N}_\Gamma=\sum_{\gamma\in\Gamma}\gamma\in\BZ[\Gamma]$ is the norm element in the group ring and $I_\Gamma=\ker(\BZ[\Gamma]\to\BZ)$ is the augmentation ideal. 
\end{remark}

\begin{example}\label{ExampleStein}

1)  Consider $H=\SL_n$ for $n>2$, with $\Gamma=\BZ/2$, acting by the involution 
\[A\mapsto  \gamma(A):=J^{-1}(A^t)^{-1}J, \quad J={\rm antidiag}(1,\ldots , 1) .
\] 
This stabilizes the upper triangular Borel and the diagonal torus $T=\{{\rm diag}(z_1,\ldots, z_n) \ |\ z_1\cdots z_n=1\}$. We have
\[
\gamma(z_1,\ldots , z_n)=(z^{-1}_n,\ldots , z_1^{-1}).
\]

\smallskip

\noindent a) {\it The odd case: $n=2m+1$. } Then
\[
\ker \big({\rm N}_\Gamma\colon T(k)\to T(k)\big)=\{(z_1,\ldots, z_m, 1, z^{-1}_m,\ldots, z_1^{-1})\}.
\]
We can write $(z_1,\ldots, z_m, 1, z^{-1}_m,\ldots, z_1^{-1})=(1-\gamma)(z_1,\ldots, z_m, (z_1\cdots z_m)^{-1}, 1,\ldots, 1))$ and so, by (\ref{H1torus}), $\RH^1(\Gamma, T(k))=(1)$, and hence $\RH^1(\Gamma, H(k))=(1)$.

\smallskip

\noindent b) {\it The even case: $n=2m$.} Then
\[
 \ker \big({\rm N}_\Gamma\colon T(k)\to T(k)\big)=\{(z_1,\ldots, z_m,  z_m,\ldots, z_1)\ |\ (z_1\cdots z_m)^2=1\}.
\]
 We have
 \[
 (1-\gamma)(a_1,\ldots, a_n)=(a_{2m}a_1, \ldots , a_{m+1}a_m,a_ma_{m+1},\ldots, a_1a_{2m}).
 \]
Since $a_1\cdots a_{2m}=1$, the product of the first $m$ entries above is equal to $1$. We can now see that the map
\[
 (z_1,\ldots, z_m,  z_m,\ldots, z_1)\mapsto z_1\cdots z_m
\]
induces $\RH^1(\Gamma, T(k))\simeq \{\pm 1\}$. Let us also consider the action of $W^\gamma$ on $\RH^1(\Gamma, T(k))$. 

We claim that there is a single $W^\gamma$-orbit in $\RH^1(\Gamma, T(k))\simeq \{\pm 1\}$. To show this, it is enough to produce one element of $W^\gamma$ that takes the neutral class $1$ to the class $-1$ which is represented by  $(1,\ldots, 1, -1, -1,1,\ldots, 1)$. Note that $\gamma$ acts on $W=S_{2m}$ as conjugation by the permutation $J=s_{1\,2m}s_{2\,2m-1}\cdots s_{m\, m+1}$. Consider $w= s_{m\, m+1}\in W^\gamma$,  which we lift to $\dot w=$   the matrix in $\SL_{2m}$ with central block $\left(\begin{matrix} 0& 1\\-1 &0\end{matrix}\right)$ and entries $1$ on the rest of the diagonal. We  now calculate that 
\[
t_w=\dot w^{-1}\gamma(\dot w)={\rm diag}(1,\ldots, 1, -1,-1, 1\ldots, 1).
\]
 This shows that $[1]\cdot w=[(1,\ldots, 1, -1, -1,1,\ldots, 1)]$ and proves our claim. Hence, we have $\RH^1(\Gamma, H(k))=(1)$.

\smallskip

2) Let us consider the following variation of b) (which will become relevant in  Example \ref{exUni} (2) below). 
Consider $H=\SL_n$ for $n=2m\geq 4$, with  $\Gamma= \BZ/2$ acting by the automorphism 
\[A\mapsto \gamma'(A):= J'^{-1}(A^t)^{-1}J', \quad J'= \epsilon\cdot J=\epsilon\cdot {\rm antidiag}(1,\ldots , 1),
\]
 where $\epsilon={\rm diag}((-1)^{(m)},  1^{(m)})\in T(k)$. For $w=s_{m\, m+1}$ and $\dot w$ as in b) above, we obtain $\gamma'(\dot w)=\dot w$, so $t_w=1$. In fact, we can see that in this case, the $W^\gamma$-action on $\RH^1(\Gamma, T(k))\simeq \{\pm 1\}$ is trivial and so there are two orbits. Hence, $\RH^1(\Gamma, H(k))\simeq \BZ/2\BZ$.
 
 This result shows that the $W^\gamma$-action on $\RH^1(\Gamma, T(k))$ is subtle and can change when we compose $\gamma$ with an inner automorphism.
 \end{example}

\subsection{Twisted forms} We place ourselves in the situation of section \ref{typesforBT}. Given a type $\tau\in \RH^1(\Spec(O');\Gamma, \CG')$, we let $h\cdot {\ome}\in \sB^e(G, F)$
be an element of the $G(F)$-orbit that is the image of $\tau$ under the bijection of Proposition \ref{LTprop2}. 

\begin{definition}\label{LTgroup}
The BT group scheme $^\tau\!\CG$ corresponding to the type $\tau$ is the unique BT group scheme 
such that
\[
^\tau\!\CG(O)=G(F)_{h\cdot{\bf a}},
\]
i.e.  $^\tau\!\CG$ is the stabilizer BT group scheme corresponding to $h\cdot{\bf a}\in \sB^e(G, F)$.
\end{definition}

Note that
this is an abuse of terminology and notation since $^\tau\!\CG$ is only determined from $\tau$ up to 
$G(F)$-conjugation. For the neutral type we can choose $h=1$ and get  the BT group scheme $\CG$.

\begin{remark}
Here is a slight variation on the definition of $^\tau\!\CG$.  Let $\theta\colon \Gamma\to \CG'(O')$ be a cocycle in the class of $\tau$. Consider
 the reductive group $^\theta G$ over $ F$ obtained by inner-twisting $G$ by $\theta$, so
\[
^\theta G( F)=\{g'\in G'(F')\ |\ \theta(\gamma)\cdot \gamma(g')\cdot \theta(\gamma)^{-1}=g',\forall\gamma\}\subset G'(F').
\]
Since $F'/F$ is tamely ramified, the building $\sB^e(^\theta G, F)$ is identified with the $\Gamma$-fixed points of the
building $\sB^e(G', F')$ with $\Gamma$-action given by the usual Galois action composed with the adjoint action via $\theta$. Since $\theta$ takes values
in $\CG'( O')$, the point $\bf a$ is also fixed by the new $\Gamma$-action which is the Galois action ``twisted" by $\theta$ as above; hence $\bf a$ also lies in $\sB^e(^\theta G, F)$.
We set
\[
^\tau\CG( O)=^\theta\!G(F)\cap \CG'( O'),
\]
which is the stabilizer of the same point ${\bf a}\in \sB^e(^\theta G, F)\subset \sB^e(G', F')$ in $^\theta G(F)$. This also defines a BT group scheme $^\theta\CG$ over $O$. 

Recall $h\in G'(F')$ with $\theta(\gamma)=h^{-1}\gamma(h)$. We now see that $g\mapsto hgh^{-1}$ gives an isomorphism of group schemes over $F$, 
\[
^\theta G\xrightarrow{\ \sim\  } G .
\]
This isomorphism extends to an isomorphism of group schemes over $O$, 
\[
^\theta\CG=^\tau\!\CG .
\]
 Hence $^\theta\CG$ is an alternative definition of $^\tau\!\CG$. 
 
 Assume $\CG'=\CG\otimes_OO'$. Then $^\theta\CG={\rm Res}_{O'/O}(\CG\otimes_OO')^\Gamma$, where the $\Gamma$-action is
 given by the usual Galois action composed with the adjoint action via $\theta$ and so
 \[
 {\rm Res}_{O'/O}(\CG\otimes_OO')^\Gamma=^\tau\!\CG.
 \] 
 The above isomorphism also appears in \cite[Thm. 2.3.1]{BalaS} and \cite[Thm. 4.1.2]{DH}, under certain assumptions, when $G'$ is constant. In these papers it appears with an alternative  description of $^\tau\!\CG$ which we explain in \S \ref{ss:BS}. 
\end{remark}

\begin{example}\label{exSL2}
Let $G=\SL_2$, let $F'/F$ be totally ramified of degree $n$ which is prime to ${\rm char}(k)$. Consider the standard apartment, which is identified with $\BR$ such that the fundamental domain for the affine Weyl group of $\SL_2(F)$ is the interval $[0, 1]$ and the fundamental domain for the affine Weyl group of $\SL_2(F')$ is the interval $[0, 1/n]$.  Let us consider ${\bf a}=1/n$. In this case  $\CG$ is an Iwahori group scheme over $O$ and $\CG'$ is a reductive group scheme over $O'$. Then any point in the building of $\SL_2(F')$ in the $\SL_2(F')$-orbit of $\bf a$ is conjugate under $\SL_2(F')$ to a  point in the standard apartment of the form $1/n+2i/n$. Since we are interested in the $\SL_2(F)$-orbits of these points, we may assume that $0< 1/n+2i/n< 1$ or $1\leq 1/n+2i/n< 2$.   The point $1$ gives a type $\tau$ with $^\tau\!\CG$ the non-standard hyperspecial parahoric group scheme, i.e., a reductive group scheme over $O$ such that $^\tau\!\CG(O)$ is not conjugate to $\SL_2(O)$ under $\SL_2(F)$. This case only occurs when $n$ is odd. For all other types $\tau$,   $^\tau\!\CG$ is an Iwahori group scheme. The points with $0< 1/n+2i/n< 1$ are conjugate to the points with $1< 1/n+2i/n< 2$.  Summarizing, we see  that there are $[\frac{n+1}{2}]$ types; if $n$ is odd, there is precisely one type with hyperspecial $^\tau\!\CG$ (which is non-standard) and if $n$ is even, there is none. 

 This number of types can also be  found easily by using (\ref{redofcoho}) and (\ref{H1Stein}).
We have $H=\bar\CG'=\SL_2$ and  the action of $\Gamma=\BZ/n$ on $H$ is trivial. Take $T=\{\diag(a, a^{-1})\,|\, a\in k^\times\}$. We have $\ker \big({\rm N}_\Gamma\big)=\{\diag(a, a^{-1})\, |\, a^n=1\}$ and  $I_\Gamma T(k)=\{1\}$.
Hence, $H^1(\Gamma, T(k))= {\rm Hom}(\BZ/n, T(k))\simeq \{a\, |\, a^n=1\}$ with the $W=\BZ/2$-action taking $a$ to $a^{-1}$. We can now see that there are $[\frac{n+1}{2}]$ orbits for the action of $W$ on $\RH^1(\Gamma, T(k))$.
\end{example}

\begin{example}\label{exUni} 
Let $G=\SU_n(\tilde F/F)$  be the quasi-split special unitary group of size $n\geq 3$ corresponding to a ramified quadratic extension $\tilde F$ of $F$. Set $F'=\tilde F$, $\Gamma={\rm Gal}(F'/F)\simeq\BZ/2$ with the non-trivial Galois automorphism denoted by $a\mapsto \bar a$. Let $\pi$ be a uniformizer of $ F'$ with $\bar \pi=-\pi$.  
Set 
\[
V'=(F')^n={\rm span}_{F'}\{ e_1,\ldots, e_n\} ,
\]
and consider the perfect $F'/F$-hermitian bilinear form $\phi: V'\times V'\to F'$ determined by $\phi(e_i, e_{j})=\delta_{i, n+1-j}$. Then
\[
G(F)=\{A\in {\rm End}_{F'}(V')\ |\ \phi(Av_1, Av_2)=\phi(v_1, v_2), \ \forall v_1, v_2\in V', \det(A)=1\}.
\]
There is a corresponding anti-involution  $A\mapsto A^*$ on ${\rm Aut}_{F'}(V')=\GL_{F'}(V')$ defined by $\phi(A v_1, v_2)=\phi(v_1, A^*v_2)$. This induces the involution ${\rm inv}: A\mapsto  (A^*)^{-1}$ and $G(F)=\{A\in {\rm Aut}_{F'}(V')\,|\, A^*A=1,\,\det(A)=1\}=\SL(V')^{{\rm inv}=1}$.

In what follows, we will use the standard description of the building $\sB({\rm SU}_n(F'/F))$ as a subset of $\sB(\SL_n, F')$ which is given by considering self-dual periodic $O'$-lattice chains, see \cite{BTcla}, \cite[4.a]{PRtwisted}. We consider two cases:
\smallskip

\noindent 1) {\it The odd case: $n=2m+1\geq 3$.} Choose a special vertex $\bf a$ in the building of $G(F)$ and consider the corresponding special maximal parahoric group scheme $\CG$. 
By \cite[4.a]{PRtwisted}, there is an $F'$-basis $\{f_i\}$ of $V'$  such that $\phi(f_i, f_{j})=\delta_{i, n+1-j}$, and such that one of the following occurs:

A) $\CG(O)$ is the stabilizer of $
\Lambda'_0={\rm span}_{O'}\{ f_1,\ldots, f_n\}$
 in $G(F)$. 
 
 B) $\CG(O)$ is the stabilizer of $
\Lambda'_m={\rm span}_{O'}\{ \pi^{-1}f_1,\ldots,\pi^{-1}f_m, f_{m+1},\ldots, f_n\}$
 in $G(F)$.

 Now set $G'={\rm SL}_{n,  F'}=\SL(V')$, and take $\CG'$ to be the hyperspecial group scheme which corresponds to the lattice $\Lambda'=\Lambda'_0$, resp. $\Lambda'_m$, i.e. with $\CG'(O')={\rm Aut}(\Lambda')\cap \SL_n(F')$. Then 
\[
G=({\rm Res}_{F'/F}\SL_{n, F'})^\Gamma,\quad  \CG=({\rm Res}_{O'/O}\CG')^\Gamma,
\]
so $\CG$ and  $\CG'$ are $\Gamma$-related. In this case, the $\SL_n(F')$-orbit of ${\bf a}$ in the building $\sB(\SL_n, F')$ consists of the vertices 
corresponding to the $O'$-lattices $\Lambda''=g\cdot \Lambda'$, $g\in \SL_n(F')$; we can see that these vertices are $\Gamma$-fixed only when the lattices are self-dual $\Lambda''^\vee=\Lambda''$ (in case A), resp. 
satisfy $\pi\Lambda''\subset\Lambda''^\vee$, with $\dim_k(\Lambda''^\vee/\pi\Lambda'')=1$ (in case B). But it follows from \cite[4.a]{PRtwisted} that all such lattices are in the $G(F)$-orbit of $\Lambda'_0$, resp. $\Lambda'_m$. Hence, in each of these cases, there is {\sl only} the neutral type. This also follows from (\ref{redofcoho}), (\ref{H1Stein}) and Example \ref{ExampleStein} (1). Indeed, in case A, the involution ${\rm inv}$ is given by $A\mapsto \gamma(A)=J^{-1}\cdot (\bar A^t)^{-1}\cdot J$, with $J={\rm antidiag}(1,\ldots, 1)$, for $A\in \SL_n(O')$.
Example \ref{ExampleStein} (1) applies to $H=\SL_{n, k}=\bar\CG'^{\rm red}$ and gives $\RH^1(\Gamma, T(k))=(1)$. The same holds in case B:  we 
again have $\RH^1(\Gamma, T(k))=(1)$.
\smallskip

\noindent 2) {\it The even case: $n=2m\geq 4$.}  Choose a special vertex $\bf a$ in the building of $G(F)$ and consider the corresponding special maximal parahoric group scheme $\CG$. By \cite[4. a]{PRtwisted}, there is an $F'$-basis $\{f_i\}$ of $V'$  such that $\phi(f_i, f_{j})=\delta_{i, n+1-j}$, and such that $\CG(O)$ is the stabilizer of 
\[
\Lambda'_m= {\rm span}_{O'}\{ \pi^{-1} f_1,\ldots, \pi^{-1} f_m,  f_{m+1},\ldots,  f_{2m}\},
\]
in $G(F)$. Notice that $\Lambda'^\vee_m=\pi^{-1}\Lambda'_m$. (There are actually two types of special vertices but they are conjugate by ${\rm U}_{2m}(F'/F)$ -- although not by ${\rm SU}_{2m}(F'/F)$, see below.) Take $\CG'$ to be the stabilizer of $\Lambda'_m$ in 
$\SL_{n, F'}$. Then  $\CG=({\rm Res}_{O'/O}\CG')^\Gamma$ and so $\CG$ and $\CG'$ are $\Gamma$-related. To write $\CG$ explicitly as the invariants of an involution on $\SL_n$, we need to express the hermitian form in an $O'$-basis of $\Lambda'_m$. In the natural $O'$-basis given above, $\phi$ is given by the
matrix 
\[
H =-\pi^{-1}\cdot {\rm diag}(-1,\ldots, -1, 1,\ldots, 1) \cdot J.
\]
The involution  is given by $A\mapsto \gamma'(A)=\bar H^{-1} (\bar A^t)^{-1}\cdot \bar H=J'^{-1}(\bar A^t)^{-1} J' $, where we set $J'={\rm diag}((-1)^{(m)}, 1^{(m)}) \cdot J$.
Then, (\ref{redofcoho}) and (\ref{H1Stein}) and the calculation in Example \ref{ExampleStein} (2) implies that there are two types, corresponding to the two orbits there. 

From the building perspective, we can see the two types as follows.  The $\SL_n(F')$-orbit of ${\bf a}$ consists of the vertices 
corresponding to $O'$-lattices $\Lambda'=g\cdot \Lambda'_m$, $g\in \SL_n(F')$; we can see that these vertices are $\Gamma$-fixed only when the lattices satisfy  $\Lambda'^\vee=\pi\Lambda'$. By \cite{PRtwisted}, such lattices are in two ${\rm SU}_{2m}(F'/F)$-orbits. The first is the orbit of ${\bf a}$, this is the neutral type. The second is the orbit under  $G(F)$ of  the
 other special vertex ${\bf a}'$, which corresponds to the lattice
\[
\Lambda'_{m'}= {\rm span}_{O'}\{ \pi^{-1} f_1,\ldots,   \pi^{-1} f_{m-1}, f_m,  \pi^{-1} f_{m+1}, f_{m+2}, \ldots,  f_{2m}\} .
\]
Note that $\Lambda'^\vee_{m'}=\pi^{-1}\Lambda'_{m'}$ and also that $f_m\mapsto \pi f_m$, $f_{m+1}\mapsto \pi^{-1}f_m$, $f_i\mapsto f_i$, for all $i\neq m$, $m+1$ gives a transformation of determinant $1$ that takes $\Lambda'_m$ to $\Lambda'_{m'}$: hence, ${\bf a}$ and ${\bf a'}$ are indeed in the same $\SL_n(F')$-orbit. 
Note also that ${\bf a}$ and ${\bf a'}$ are in the same ${\rm U}_n(F'/F)$-orbit (since $f_m\mapsto f_{m+1}$, $f_{m+1}\mapsto f_m$, $f_i\mapsto f_i$, for $i\neq m$, $m+1$, is a unitary transformation) but {\sl not} in the same ${\rm SU}_n(F'/F)$-orbit. The stabilizer  of the lattice $\Lambda'_{m'}$ in the special unitary group is the twisted parahoric group $G(F)_{\bf a'}$ associated to the non-neutral type.
 \end{example}

 \begin{remark}
 In the even case $n=2m$, we can also choose $\CG$ to be the stabilizer of the self-dual lattice $\Lambda'_0=
 {\rm span}_{O'}\{ f_1, f_2, \ldots,  f_{2m}\}$ in the basis above and $\CG'=\SL(\Lambda'_0)$.
 Then, using Example \ref{ExampleStein} (1b), we see that there is only the neutral type.
\end{remark}
 
 \subsection {Relation to \cite{BalaS} and \cite{DH}}\label{ss:BS}

In \cite[\S 2.3]{BalaS}, Balaji-Seshadri  (see also Damilioni-Hong \cite[\S 4]{DH}) consider the case where $H=\bar\CG'^{\rm red}$ is connected and the action of $\Gamma$ on $H(k)$ is trivial. In this case,  
\begin{equation}\label{eqH1}
\begin{aligned}
\RH^1(\Gamma, H(k))&=\RH^1(\Gamma, T(k))/W={\rm Hom}(\Gamma, T(k))/W\\
&={\rm Hom}(\mu_e, T)/W= {\rm Hom}(X^*(T), \BZ/e\BZ)/W,
\end{aligned}
\end{equation}
cf. \eqref{H1Stein}.
The injection $\BZ/e\BZ\subset \BQ/\BZ$ given by $1\mapsto e^{-1}\, {\rm mod}\, \BZ$ induces the injection 
\[
 {\rm Hom}(X^*(T), \BZ/e\BZ)\subset 
{\rm Hom}(X^*(T), \BQ/\BZ)={\rm Hom}(X^*(T), \BQ)/X_*(T).
\]

Assume now that $H$ is semi-simple and simply connected over $k$ and consider $\CG=H\otimes_k O$, $\CG'=H\otimes_kO'$. In this case, Balaji-Seshadri give an alternative way to describe the twisted BT-group associated to a local type. Indeed, using the affine Weyl group  $W_{\rm aff}=W\rtimes X_*(T)$, and combining with \eqref{eqH1}, we see that
\[
\RH^1(\Gamma, H(k))= 
{\rm Hom}(X^*(T), \BZ/e\BZ)/W\subset {\rm Hom}(X^*(T), \BR)/W_{\rm aff}.
\]
Now ${\rm Hom}(X^*(T), \BR)$ is the apartment corresponding to the torus $T\otimes_k F$ in the building of $H\otimes_k F$ over $F$. Hence 
we can use the above to identify $\RH^1(\Gamma, H(k))$ with a subset of a fixed alcove in that building.
In particular, each local type $\tau\in \RH^1(\Gamma, H(k))$ gives a well-defined point $[\tau]$ in the fixed alcove. The stabilizer of this point defines a parahoric group scheme for $H\otimes_k F$.
We claim that  this parahoric agrees up to conjugation by $H(F)$ with the parahoric group scheme   $^\tau\!\CG$  associated to the local type $\tau$. 

Recall the maximal torus $T\subset H$ over $k$; then $T\otimes_k F$ is a maximal torus of $G=H\otimes_k F$, and the hyperspecial point ${\bf a}\in \sB(G, F)$ that corresponds to $\CG(O)=H\otimes_k O$ is a base point $0$ of the apartment $A(T\otimes_k F)=X_*(T)\otimes_\BZ \BR$ for $T\otimes_k F$. Suppose the local type $\tau$ is given by a cocycle $\theta: \Gamma\to \CG'(O')=H(O')$. Denote by $\bar\theta: \Gamma \to H(k)$ the reduction of $\theta$ modulo $\varpi'$.
We can replace $\theta$ by a cohomologous cocycle $\theta'$ with the property that $\bar\theta'$ takes values in the $k$-points of $T$. Then, by lifting, we find $\theta'': \Gamma\to T(O')$ which is cohomologous to $\theta'$. Hence, replacing $\theta$ by $\theta''$, we may assume from the start that the local type is  given by a cocycle $\theta: \Gamma\to T(O')$. Let $h=t\in T(F')$ such that $\theta(\gamma)=t^{-1}\gamma(t)$ and  take
\[
^\tau\!\CG(O)=G(F)_{t\cdot {\bf a}}=G(F)\cap t H(O') t^{-1}.
\]
The point $t\cdot {\bf a}$ lies in the apartment $A(T\otimes_k F)=X_*(T)\otimes_\BZ \BR$ and it remains to
see that its $W_{\rm aff}$-orbit agrees with the $W_{\rm aff}$-orbit of the point $[\tau]$ above.

Let us make explicit  the isomorphism $\Gamma= \mu_e$. Choose uniformizers $\varpi\in O$ and $\varpi'\in O'$ such that $\varpi'^e=\varpi$. Then $\gamma(\varpi')=\zeta(\gamma)\varpi'$ for $\zeta(\gamma)\in\mu_e$, and the isomorphism is given by $\gamma\mapsto \zeta(\gamma)$.  Choose an isomorphism $T\simeq \BG_m^r$ and a primitive $e$-th root of unity $\zeta=\zeta_e\in k^*$ which also determines a generator $\gamma_\zeta\in \Gamma$. Suppose that $\bar\theta: \Gamma\to T(k)=(k^*)^r$ is given by $\bar\theta(\gamma_\zeta)=(\zeta^{a_1},\ldots, \zeta^{a_r})$, $0\leq a_i<e$. 
Then we can take $t=(\varpi'^{a_i},\ldots, \varpi'^{a_r})\in T(F')=(F'^\times)^r$; the corresponding translation element in $X_*(T)\otimes_\BZ\BR=\BR^r$ is $(a_1/e,\ldots, a_r/e)$ and $t\cdot {\bf a}=(a_1/e,\ldots, a_r/e)+0$. This has the same $W_{\rm aff}$-orbit as the point $[\tau]$ above.

\section{Global $(\CG',\Gamma)$-bundles and their local types}\label{s:locglob}

We now return to the global set-up. Let $X'/X$ be a tame Galois cover with Galois group $\Gamma$. Let $G$ be a reductive group over $K$ and $G'$ its base change over $K'$. Let $\CG$, resp. $\CG'$, be $\Gamma$-related stabilizer BT group schemes for $G$, resp. $G'$, corresponding  to the collection of points $\{{\bf a}_x\in\sB^e(G, K_{x})\mid x\in X\}$. Recall that then $\CG=\Res_{X'/X}(\CG')^\Gamma$, cf. Definition \ref{defrelated}.

\subsection{A Leray exact sequence} Recall from section \ref{ggamma} the stack-theoretic interpretation of $(\CG', \Gamma)$-bundles. 
It uses the quotient stack $\fkX:=[X'/\Gamma]$ which supports a natural  morphism 
\begin{equation}
q: \fkX=[X'/\Gamma]\to X .
\end{equation}
 Recall also the natural morphism
$\fkG:=[\CG'/\Gamma]\to \fkX=[X'/\Gamma]$ which is representable by a (relative) smooth affine group scheme. The groupoid of $(\CG',\Gamma)$-bundles over $X'$ 
is naturally equivalent to the groupoid of $\fkG$-bundles (i.e. $\fkG$-torsors) over $\fkX$, cf. loc.~cit.
We   now consider the corresponding set of isomorphism classes
\[
\RH^1(\fkX, \fkG)\cong \RH^1(X'; \Gamma, \CG').
\]
(Here we can use the \'etale topology.)  This is the global analogue of the set appearing in section \ref{s:localt}.  

\begin{proposition}\label{propLeray}
There is an exact sequence of pointed sets  
 \begin{equation}\label{eqleray1}
0\to \RH^1(X, \CG)\to \RH^1(X'; \Gamma, \CG' )\xrightarrow{\ {\rm LT}\ }  \prod_{x\in B}  \RH^1(I_{x'},  \bar\CG'^{\rm red}_{x'}(\bar k))^{{\rm Gal}(\bar k/k(x))}.
\end{equation}
Here, $B\subset X$ is the (finite) branch locus of $\pi: X'\to X$ and, for $x\in B$ a lift $x'$ of $x$ to $X'$ is chosen;  $\bar\CG'^{\rm red}_{x'}$ denotes the maximal reductive quotient of the special fiber of $\CG'$ over $x'$. Also,  $I_{x'}\subset \Gamma$ is the inertia subgroup. 
\end{proposition}

If $k$ is algebraically closed, then (\ref{eqleray1}) is an exact sequence of pointed sets
 \begin{equation}\label{eqleray4}
0\to \RH^1(X, \CG)\to \RH^1(X'; \Gamma, \CG' )\xrightarrow{\  {\rm LT}\ }   \prod_{x\in B}  \RH^1(\Gamma_{x'},  \bar\CG'^{\rm red}_{x'}(k)),
\end{equation}
where $\Gamma_{x'}=I_{x'}$ is the stabilizer of $x'\in X'$.
\begin{remark}
We note that an analogue (in his context) of this exact sequence appears in Grothendieck's Bourbaki talk (\cite{GrothBour}, display (II), p. 59). He also describes in terms of automorphism sheaves the fibers of the second map, comp. Proposition \ref{LTfiber} b) below. However, Grothendieck's expressions for the terms on the right in \eqref{eqleray4} and the automorphism sheaves is less precise.
\end{remark}

\begin{proof}
 By \cite[Chap. V, Prop. 3.1.3]{Giraud}, there is an exact sequence of pointed sets
\begin{equation}\label{eqleray3}
0\to\RH^1(X, R^0q_*\fkG)\to \RH^1(\fkX, \fkG)\to \RH^0(X, R^1q_*\fkG)
\end{equation}
where the derived images $R^iq_*$ are for $q=q_{\rm et}\colon \fkX_{\rm et}\to X_{\rm et}$. (This can be viewed as a non-abelian ``Leray type'' sequence, comp. \cite[\S 1]{Douai}). Hence, it remains to determine $R^iq_*(\fkG)$ for $i=0,1$.

Since ${\rm Res}_{X'/X}(\CG')(S)=\CG'(X'\times_XS)$, for any $X$-scheme $S$, we have
\[
R^0q_*(\fkG)=\RH^0(\Gamma, {\rm Res}_{X'/X}(\CG'))=({\rm Res}_{X'/X}(\CG'))^\Gamma=\CG,
\]
where we denote also by $\CG$ the \'etale group sheaf over $X$
corresponding to the  group scheme $\CG$.  

Let $x'$ and $x$ be points of $X'$ and $X$, with $\pi(x')=x$.
Denote by $\Gamma_{x'}\subset \Gamma$ the stabilizer (decomposition) subgroup of $x'$ and recall  the inertia subgroup $I_{x'}\subset \Gamma_{x'}$. 
Set $O^h_{X, x}(x')=(O^h_{X', x'})^{I_{x'}}$, so that $O^h_{X, x}(x')/O^h_{X, x}$ is an unramified Galois extension with group $\Gamma_{x'}/I_{x'}={\rm Gal}(k(x')/k(x))$.
Then there is a $\Gamma$-equivariant isomorphism of algebras, (\cite{Raynaud})
\[
O_{X'}\otimes_{O_X}O^{h}_{X, x}(x')\simeq {\rm Ind}_{I_{x'}}^\Gamma O^{h}_{X', x'} .
\]
This gives
\[
[X'/\Gamma]\times_X \Spec(O^{sh}_{X, x})\simeq   [\Spec(O^{sh}_{X', x'})/I_{x'}].
\]
Hence,   the stalk  of the \'etale sheaf
$R^1q_*\fkG$ at the geometric point $\bar x$ given by $k(x)\subset k(x')\subset \bar k=k(\bar x)$ is  the ${\rm Gal}(\bar k/k(x))$-set
\[
\RH^1(X'\times_X\Spec(O^{sh}_{X, x}); \Gamma, \CG')\simeq \RH^1(\Spec(O^{sh}_{X',x'}); I_{x'}, \CG').
\]
Since $\CG'$ is smooth,  all $\CG'$-torsors over $\Spec(O^{sh}_{X', x'})$ are trivial and by Proposition \ref{LTprop}
\[
\RH^1(\Spec(O^{sh}_{X',x'}); I_{x'}, \CG')\simeq \RH^1(I_{x'}, \CG'(O^{sh}_{X', x'})).
\]
In particular, if $x$ lies on the complement $U=X\setminus B$ of the branch locus $B$ of $X'\to X$, then  
$(R^1q_*\fkG)_{\bar x}=(0)$. Artin approximation implies that the injection $O^{sh}_{X', x'}\hookrightarrow O'_{\bar x'}$ induces a bijection
\begin{equation*}\label{Artin}
\RH^1(I_{x'}, \CG'(O^{sh}_{X', x'}))\simeq \RH^1(I_{x'}, \CG'(O'_{\bar x'})).
\end{equation*}
Recall that we assume that the  cover $X'\to X$ is tamely ramified, so all the inertia groups $I_{ x'}$ have order prime to the characteristic of $k$.  
Therefore, using that the kernel of $\CG'_{x'}\to   \bar\CG'^{\rm red}_{x'}$ is pro-unipotent, we obtain 
\[
\RH^1(I_{x'}, \CG'(O^{sh}_{X', x'}))\simeq \RH^1(I_{x'}, \CG'(O'_{\bar x'}))\simeq \RH^1(I_{x'}, \bar\CG'^{\rm red}_{x'}(\bar k)).
\] 
This  completes the proof of the proposition.
\end{proof}

\begin{definition}\label{localTypeDef}
Let $\CP'$ be a $(\CG',\Gamma)$-bundle over $X'$. The cohomology class 
\[
{\rm LT}_{x'}(\CP')\in \RH^1(I_{x'},  \bar\CG'^{\rm red}_{x'}(\bar k))= \RH^1(I_{x'}, \CG'(O'_{\bar x'}))
\]
 is called  the \emph{local type}
of $\CP'$ at $x'$. 
\end{definition}

 Note that if $x'_i$, $i=1$, $2$, are two points of $X'$ both mapping to $x$, there is $\gamma\in \Gamma$ with $\gamma(x'_1)=x'_2$ and we can use  a lift of $\gamma$ to identify $\RH^1(I_{x'_1},  \bar\CG'^{\rm red}_{x'_1}(\bar k))\simeq \RH^1(I_{x'_2},  \bar\CG'^{\rm red}_{x'_2}(\bar k))$. The maps ${\rm LT}_{x'_1}$ and ${\rm LT}_{x'_2}$ are compatible with such an identification. We will often abuse notation and write ${\rm LT}_{x}$ instead of ${\rm LT}_{x'}$. 

\begin{remark}\label{remLTfinite}
If $k$ is a finite field and the fiber of $\CG'$ at $x'$ is connected, then every $\CG'$-torsor is trivial over $\Spec(O^h_{X', x'})$ by the smoothness of $\CG'$ and Lang's theorem. Under these assumptions, a similar construction using restriction to $\Spec(O^h_{X', x'})$ allows us to define the local type of a $(\CG',\Gamma)$-bundle $\CP'$ over $X'$ at $x'$ as an element in $\RH^1(I_{x'}, \bar\CG'^{\rm red}_{x'}(k(x')))$. This maps to the local type ${\rm LT}_{x'}(\CP')\in \RH^1(I_{x'},  \bar\CG'^{\rm red}_{x'}(\bar k))$  given above.
\end{remark}

\subsection{Bundles with fixed local type (algebraically closed ground field).}\label{fixedalclosed}
Let us now assume that $k$ is algebraically closed.  
Under these assumptions, we will study
the fibers of the map 
\begin{equation}\label{globalLT}
{\rm LT}: \RH^1([X'/\Gamma], [\CG'/\Gamma])\to \prod_{x\in B}  \RH^1(\Gamma_{x'},  \CG'(O'_{x'}))=\prod_{x\in B}\RH^1(\Gamma_{x'}, \bar\CG'^{\rm red}_{x'}(k)) .
\end{equation}

\begin{proposition}\label{LTfiber}
a) The map \eqref{globalLT} is surjective.  

\smallskip

b) Let $\tau\in \prod_{x\in B}\RH^1(\Gamma_{x'}, \bar\CG'^{\rm red}_{x'}(k))$. Fix a $(\CG', \Gamma)$-bundle $\CP'$ with $LT(\CP')=\tau$, and let $\CG^\flat$ be the corresponding  ``twist'' of $\CG$ so that there 
 is a bijection between the set of isomorphism classes of $(\CG',\Gamma)$-bundles over $X'$ with local type $\tau$ and the set of isomorphism classes of $\CG^\flat$-bundles over $X$, i.e., there is a bijection
 \[
{\rm LT}^{-1}(\tau)\simeq \RH^1(X, \CG^\flat) ,
\]
cf. \cite[ Ch. V, 3.1.4, 3.1.6]{Giraud}. 

Then $\CG^\flat$ is a stabilizer  BT group scheme $^\tau\!\CG$ for $G$  such that for $x\notin  B$, there is an isomorphism $^\tau\!\CG_x\simeq\CG_x$ and such that, for $x\in B$, denoting the component of $\tau$ at $x$ by $\tau_x$, the group scheme 
$^\tau\!\CG_x$ is isomorphic to the twisted form $^{\tau_x}\CG_x$ of $\CG_x$ (cf. Definition \ref{LTgroup} for the definition of this BT group scheme over $\Spec(O_x)$). 
\end{proposition}

\begin{proof} 
a)  Starting from $\tau=(\tau_{x})_{x\in B}$, we will create a $\CG'$-torsor over $X'$ by glueing \`a la Beauville-Laszlo the trivial $\CG'$-torsors over $U'=\pi^{-1}(U)$ and over $\sqcup_{x'\in X'\setminus U' }\Spec( O'_{x'})$ by using inner automorphisms given by elements $h_{x'}\in G'(K'_{x'})$. For each $x\in B$ choose a preimage $x'$, and represent the class $\tau_x$ by a cocycle $\theta\colon \Gamma\to \CG'(O'_{x'})$.  Then by section \ref{ss:types}, we can write $\theta(\gamma)=h^{-1}\gamma(h)$, for some $h=h_{x'}\in G'(K'_{x'})$.   Write  $\{x'=x'_1, x'_2,\ldots, x'_f\}$ for the $f=[\Gamma: \Gamma_{x'}]$ points on $X'$ above $x$ and choose $\gamma_i\in \Gamma$ such that 
$x'_i=\gamma_i(x')$. Then   set $h_{x'_i}=\gamma_i (h_{x'})$. The condition that $h^{-1}_{x'}\gamma (h_{x'})\in \CG'( O'_{x'})$, for all $\gamma\in \Gamma_{x'}$, implies that the  $\CG'$-torsor over $X'$ defined by gluing via $\{h_{x'}\mid x'\in X'\setminus U'\}$ acquires a compatible $\Gamma$-structure and so it is a $(\CG', \Gamma)$-torsor. It has the local type $\tau$ by construction.

b) Fix a  $(\CG', \Gamma)$-bundle $\CP'$
whose isomorphism class lies in ${\rm LT}^{-1}(\tau)$. By \cite[Ch. V, 3.1.4, 3.1.6]{Giraud}, ${\rm LT}^{-1}(\tau)$ is in bijection with the set of isomorphism classes of $\CG^\flat$-bundles
over $X$, where $\CG^\flat$ is the twist 
\[
\CG^\flat:=R^0q_*({\rm Aut}(\CP'))=({\rm Res}_{X'/X}{\rm Aut}(\CP'))^\Gamma.
\]
 We now prove that $\CG^\flat$ has the asserted properties. By Steinberg's theorem, the $(\CG',\Gamma)$-bundle $\CP'$
is generically trivial. Hence there exists an affine open subset $U_{\CP'}\subset X$ such that the restriction of $\CP'$ to $\pi^{-1}(U_{\CP'})$ is trivial. 
 Choose a trivialization, i.e. a $(\CG',\Gamma)$-isomorphism 
\[
a_{U_{\CP'}}: \CP'\times_{X'}\pi^{-1}(U_{\CP'})\simeq \CG'\times_{X'}\pi^{-1}(U_{\CP'})
\]
 which also gives $\CG^\flat\times_X U_{\CP'}\simeq \CG\times_XU_{\CP'}$. Let $x'\in X'$ with $\pi(x')=x$ and choose 
a section 
\[
a_{x'}: \CP'\times_{X'} \Spec(O'_{x'})\simeq \CG'\times_{X'}\Spec(O'_{x'}),
\]
and set $h_{x'}=({a_{U_{\CP'}}}_{|\Spec K'_{x'}}\cdot {a_{x'}}_{|\Spec K'_{x'}}^{-1})(1)\in G'(K'_{x'})$. Then we can identify
\[
\CP'  (O'_{x'})=\CG' (O'_{x'})\cdot h^{-1}_{x'}\subset G'(K'_{x'}).
\]
For $\gamma\in \Gamma_{x'}$, we have $h_{x'}^{-1}\gamma(h_{x'})\in \CG'(O'_{x'})$; the local type of $\CP'$ at $x'$ is given by the class of the 
$1$-cocycle  $\gamma\mapsto h_{x'}^{-1}\gamma(h_{x'})$.  If $x\in B$, then, since the local type of $\CP'$ at $x'$ is $\tau_x$, the map $\gamma\mapsto h_{x'}^{-1}\gamma(h_{x'})$ gives a $1$-cocycle in the cohomology class of $\tau_x$. Similarly, if $x\notin B$, then, since the local type of $\CP'$ at $x'$ is trivial, the map $\gamma\mapsto h_{x'}^{-1}\gamma(h_{x'})$ gives the trivial   cohomology class.
We can now check that the completions of $\CG^\flat$ and $^\tau\!\CG$ agree at all points, i.e.  that there  are group scheme isomorphisms $\CG^\flat_x\simeq\ ^{\tau_x}\CG_x$, for all $x\in X$, which are compatible with $\CG^\flat\times_XU_{\CP'}\simeq \CG\times_XU_{\CP'}=\, ^\tau\!\CG\times_XU_{\CP'}$. Let $x\in B$. It is enough to check agreement of the point sets $\CG^\flat(O_{ x})$ with  $^{\tau_x}\CG_x( O_{ x})$ under our identifications. On the one hand
\[
\CG^\flat (O_x)={\rm Aut}(\CP'(O_x))^\Gamma={\rm Aut}(\CP'(O'_{x'}))^{\Gamma_{x'}}={\rm Aut}(\CG'(O'_{x'})\cdot h^{-1}_{x'})^{\Gamma_{x'}}.
\]
On the other hand, $^{\tau_x}\CG_x(O_x)$ is the stabilizer of $h_{x'}\cdot {\bf a_x}$ in $G(K_x)$, cf. Definition \ref{LTgroup}. In particular, we have
\begin{equation}\label{twistInv}
^{\tau_x}\CG(O_x)=\big(h_{x'}\CG'(O'_{x'})h_{x'}^{-1}\big)^{\Gamma_{x'}},
\end{equation}
which proves the claim.  For $x\notin B$ a similar argument applies. 
\end{proof}

\begin{remark}\label{extraInfo}
In part (b) and when $B\neq \emptyset$, the proof actually gives the additional information that the restrictions of $\CG^\flat$ and $\CG$ over $U=X\setminus B$ are isomorphic group schemes. This is because the $\CG_U$-bundle $({\rm Res}_{X'/X}(\CP')_{|U})^\Gamma$ is trivial by \cite[Thm 4]{Hein}, so we can take $U_{\CP'}=U$ in the argument. If $B=\emptyset$, i.e. the cover $X'\to X$ is unramified, then $\CG^\flat={\rm Aut}(\CP)$, the strong inner form of $\CG$ defined by the $\CG$-torsor $\CP=({\rm Res}_{X'/X}(\CP'))^\Gamma$.
\end{remark}

\begin{remark}
It should be pointed out that the notation $^\tau\!\CG$ is somewhat misleading, since the group scheme $\CG^\flat$ depends on $\CP'$. Relatedly, the group scheme $^\tau\!\CG$ does not seem to be determined (up to isomorphism) by its above description through its localizations at all points $x\in X$. 
\end{remark}

\begin{remark}\label{remHar}
Let us collect  here some facts on the triviality of $\CG$-bundles. 

\begin{altenumerate}
\item First, there is Heinloth's result \cite[Thm 4]{Hein}: it states for a parahoric BT group scheme $\CG$ on $X$ for a simply connected group $G$, that an $S$-family of  $\CG$-bundles on $X$  becomes trivial on $(X\setminus \{x\})\times S'$ after an \'etale extension $S'\to S$. If $G$ is semi-simple but not necessarily simply connected, one has to allow an fppf-base change $S'\to S$ ( \cite[Thm 5]{Hein}).

\item  For the next result, let us introduce the following terminology, cf. \cite{HInv}.   Let $\CG$ be a reductive group scheme over a Dedekind ring $A$ with semi-simple simply connected generic fiber. A $\CG$-bundle on  $\Spec(A)$ is \emph{rationally quasi-trivial} if the associated strong inner form of $\CG$  over the fraction field is quasi-split. Then a rationally quasi-trivial $\CG$-bundle is trivial, cf. \cite[Satz 3.3]{HInv}.
 
\item Let $A$ be the affine ring of a proper open subset $U$ of a curve $X$ over a field $k$.  If $k$ is algebraically closed, then by Steinberg's theorem the hypothesis of rational quasi-triviality in Harder's result  above is automatic.  When $k$ is finite, this is not true and the triviality of $\CG$-bundles over $\Spec(A)$ in ii) can fail if $G\otimes_K K_x$  is anisotropic for all $x\in X\setminus U$.

\item Let $X$ be a curve over an algebraically closed field $k$. Let $\CG$ be a parahoric BT group scheme over $X$ for a semi-simple simply connected group $G$ over $k(X)$. Let $U$ be an open subset such that $\CG_{|U}$ is a reductive group scheme over $U$. Then a $\CG$-torsor over $U$ extends to a $\CG$-torsor over $X$. Indeed, the $\CG$-torsor on $U$ is trivial. The  analogous statement also holds when $k$ is finite, even though the $\CG$-torsor on $U$ may not be trivial. Indeed, this follows by  Beauville-Laszlo glueing at the missing points $x\in X\setminus U$, using  $\RH^1(K_x, G)=0$.
\end{altenumerate}
\end{remark}

\subsection{Bundles with fixed local type (finite ground field).}\label{fixedfinite}

In this subsection we assume that the ground field $k$ is finite, and indicate a variant of the previous subsection. Recall from Remark \ref{remLTfinite} the definition of the local type ${\rm LT}_{x'}(\CP')\in \RH^1(I_{x'}, \bar\CG'^{\rm red}_{x'}(k(x')))$ of the  $(\CG', \Gamma)$-bundle $\CP'$. We obtain the map analogous to \eqref{globalLT},
\begin{equation}\label{globalLTfinite}
{\rm LT}: \RH^1([X'/\Gamma], [\CG'/\Gamma])\to\prod_{x\in B}\RH^1(I_{x'}, \bar\CG'^{\rm red}_{x'}(k(x'))) .
\end{equation}
 The analogue of Proposition \ref{LTfiber} is the following statement.
 
\begin{proposition}\label{LTfiberfinite}
a) The map \eqref{globalLTfinite} is surjective.  

\smallskip

b) Let $\tau\in \prod_{x\in B}\RH^1(I_{x'}, \bar\CG'^{\rm red}_{x'}(k(x')))$. Fix a $(\CG', \Gamma)$-bundle $\CP'$ with $LT(\CP')=\tau$, and let $\CG^\flat$ be the corresponding  strong inner form of $\CG$ so that there 
 is a bijection between the set of isomorphism classes of $(\CG',\Gamma)$-bundles over $X'$ with local type $\tau$ and the set of isomorphism classes of $\CG^\flat$-bundles over $X$, i.e., there is a bijection
 \[
{\rm LT}^{-1}(\tau)\simeq \RH^1(X, \CG^\flat) ,
\]
cf. \cite[ Ch. V, 3.1.4, 3.1.6]{Giraud}. 

Assume that $G$ and $G'$ are semi-simple and simply connected. Then $\CG^\flat$ is a parahoric  BT group scheme $^\tau\!\CG$ such that for $x\notin  B$, there is an isomorphism $^\tau\!\CG_x\simeq\CG_x$ and such that, for $x\in B$, denoting the component of $\tau$ at $x$ by $\tau_x$, the group scheme 
$^\tau\!\CG_x$ is isomorphic to the twisted form $^{\tau_x}\CG_x$ of $\CG_x$. 
\end{proposition}

\begin{proof} Here the definition of  $^{\tau_x}\CG_x$  is analogous to Definition \ref{LTgroup}: In Definition \ref{LTgroup}, Steinberg's theorem (in the local situation) enters, which allows us to trivialize the cocycle $\theta:\Gamma\to \CG'(O'_{x'})$ corresponding to $\tau_x$, and construct the element $h$ entering in the definition of $^{\tau_x}\CG_x$ in loc.~cit. Here, instead of Steinberg's local theorem, we use Remark \ref{remfinitevar}. In fact, replacing  in the arguments in the proof of Proposition \ref{globalLT}, the local Steinberg theorem by  Remark \ref{remfinitevar}, and the global Steinberg theorem by Harder's theorem \cite{HCrel} that $\RH^1(K, G)=0$ when $G$ is semi-simple simply connected, the proof goes through with no changes.  
\end{proof} 

 \begin{remark} 
Note that Remark \ref{extraInfo} does not extend to the case when $k$ is a finite field: if $B\neq \emptyset$, we do not know that 
$\CG^\flat_U\simeq \CG_U$  in general.  This is because of the possible existence of non-trivial $\CG$-torsors over $U$, see Remark \ref{remHar} (iii).
\end{remark}

\section{Moduli of $\CG$-bundles and of $(\CG', \Gamma)$-bundles}\label{s:moduli}

 We denote by $\Bun_\CG$ the stack of $\CG$-bundles on $X$. For any smooth affine group scheme $\CG$ over $X$, this is an algebraic stack which is smooth over $\Spec(k)$, cf.  \cite[Prop. 1]{Hein}. 

In the context of Definition \ref{GGammadef}, we also have the stack $ \Bun_{(\CG', \Gamma)}$ of $(\CG', \Gamma)$-bundles on $X'$.  Note that the extra datum that defines a $(\CG',\Gamma)$-structure on a $\CG'$-bundle $\CP'$ consists of $\CG'$-automorphisms $\phi_\gamma: \gamma^*(\CP')\xrightarrow{\sim} \CP'$ between  $\CP'$ and its pull-backs $\gamma^*(\CP')$, for all $\gamma\in \Gamma$, which satisfy certain compatibilities. From the representability of $\Bun_{\CG'}$, it now follows with standard arguments that $ \Bun_{(\CG', \Gamma)}$ also defines an algebraic stack.

\begin{theorem}\label{thmlocconst}
Let $k$ be algebraically closed. Let $X'/X$ be a finite tame Galois covering with Galois group $\Gamma$. Let $\CG$ and $\CG'$ be $\Gamma$-related  stabilizer Bruhat-Tits group schemes over $X$, resp. $X'$.

\noindent a) The local type map
\[
{\rm LT}\colon  \Bun_{(\CG', \Gamma)}(k)\to \prod_{x\in B} \RH^1({\Gamma_{x'}}, \CG'( O'_{x'}))
\] is locally constant.

\noindent b) For each element $\tau\in   \prod_{x\in B} \RH^1({\Gamma_{x'}}, \CG'(O'_{x'}))$, let $(\Bun_{(\CG', \Gamma)})_\tau$ be the open and closed substack where the value of ${\rm LT}$ equals $\tau$. Fix a point in $(\Bun_{(\CG', \Gamma)})_\tau(k)$. Then there is a canonical isomorphism of algebraic stacks preserving the base points
\[
(\Bun_{(\CG', \Gamma)})_\tau\simeq \Bun_{\, ^\tau\!\CG} .
\]
Here $^\tau\!\CG$ is the stabilizer BT group scheme on $X$ from Proposition \ref{LTfiber}.
\end{theorem}

The theorem implies immediately

\begin{corollary}
In the situation of Theorem \ref{thmlocconst}, fix for each $\tau\in   \prod_{x\in B} \RH^1({\Gamma_{x'}}, \CG'( O'_{x'}))$ an element $\CP_\tau$ with ${\rm LT}(\CP_\tau)=\tau$, with associated stabilizer BT group scheme $^\tau\!\CG$. Then there is a natural isomorphism
\[
\Bun_{(\CG', \Gamma)}\simeq \bigsqcup_{\tau\in   \prod_{x\in B} \RH^1({\Gamma_{x'}}, \CG'( O'_{x'}))}\Bun_{\,^\tau\!\CG} .
\]\qed
\end{corollary}
As another consequence, we note.
\begin{corollary}
The algebraic stack $\Bun_{(\CG', \Gamma)}$ is smooth over $\Spec(k)$. If $G$ and $G'$ are semi-simple and simply connected, each summand $(\Bun_{(\CG', \Gamma)})_\tau$ is connected and hence
\[ 
\pi_0(\Bun_{(\CG', \Gamma)})=\prod_{x\in B} \RH^1({\Gamma_{x'}}, \CG'(O'_{x'})) .
\]
\end{corollary}
\begin{proof}
The first assertion follows from the smoothness of $\Bun_\CG$ for any smooth affine group scheme $\CG$ over $X$, cf. above. The second assertion follows from Heinloth's theorem that $\Bun_\CG$ is connected for any parahoric BT group scheme with semi-simple simply connected generic fiber, cf. \cite[Thm. 2]{Hein}. 
\end{proof}

\begin{proof} (of Theorem \ref{thmlocconst}) 
By Theorem \ref{thmpur}, the natural map
\[
{\rm Bun}_{\CG}\to {\rm Bun}_{(\CG',\Gamma)}
\]
 is an open and closed embedding. On the other hand, by Proposition \ref{propLeray}, a point in ${\rm Bun}_{\CG',\Gamma}(k)$ lies in $\Bun_\CG(k)$ if and only if it is mapped under the local type map to the distinguished element $\tau_0$. Hence ${\rm LT}^{-1}(\tau_0)$ is open and closed in $\Bun_{(\CG', \Gamma)}$. Now let $\tau$ be arbitrary, and fix a point $\CP\in{\rm Bun}_{\CG',\Gamma}(k)$ mapping to $\tau$ under ${\rm LT}$, cf. Proposition \ref{LTfiber}. Recall that associated to $\CP$ is a stabilizer BT group scheme $^\tau\!\CG$ over $X$ with generic fiber $G$ such that the associated parahoric BT group scheme on $X'$ is identified with $\CG'$, compatibly with its $\Gamma$-action, cf.  Proposition \ref{LTfiber}. Furthermore, a point in ${\rm Bun}_{(\CG',\Gamma)}(k)$ lies in $\Bun_{^\tau\!\CG}(k)$ if and only if it is mapped under the local type map to  $\tau$.  Applying Theorem \ref{thmpur} to $^\tau\!\CG$ instead of $\CG$ implies that ${\rm LT}^{-1}(\tau)$ can be identified with the $k$-valued points of the open and closed substack $\Bun_{\,^\tau\!\CG}$ of $\Bun_{(\CG', \Gamma)}$. Theorem \ref{thmlocconst} is proved.
\end{proof}

\section{Concluding remarks}\label{s:concl}
The main thrust of the whole paper is to represent a BT group scheme on the curve $X$ as $\CG=\Res_{X'/X}(\CG')^\Gamma$, where $\CG'$ is a reductive group scheme over the tame cover $X'$ with compatible $\Gamma$-action. In this spirit it seems natural to express objects associated with $\CG$ in terms of $\CG'$ with its $\Gamma$-action. We did this here for the theory of $\CG$-bundles and their moduli stacks. But there are more objects associated to $\CG$ (or $G$),  which one might view with  profit from this point of view. We collect in this section a few (half-baked?) suggestions for further work inspired by this point of view. In this section, for simplicity, we make the blanket assumption that $G$ is a semi-simple simply connected group and denote by $\CG$ a parahoric BT group scheme for $G$ over $X$. 

\subsection{Uniformization}
Recall that Heinloth \cite{Hein} has proved the uniformization theorem for $\CG$-bundles, as conjectured in \cite[Conj. 3.3]{PRq}: for any $x\in X$, there is a presentation
\begin{equation}
\Bun_\CG \simeq \Gamma_{X\setminus \{x\}}(\CG)\backslash LG_x/L^+\CG_x .
\end{equation}
In light of a presentation of $\CG$ in terms of a tame Galois cover $X'/X$ and a reductive group scheme $\CG'$ over $X'$, it seems reasonable to connect this presentation to the uniformization theorem for $\CG'$ which is more standard, cf. \cite{PRtwisted}. Note that in the context of \cite{HoKu}, i.e., when $\CG'=H\times_{\Spec(k)} X'$ for a $k$-group $H$ equipped with an action $\rho\colon \Gamma\to\Aut(H)$, Hong and Kumar consider the stack of \emph{quasi-parabolic $\CG$-bundles}  (these are $\CG$-bundles with an additional (``parabolic'') structure) and prove an analogous uniformization theorem, where one divides out on the left by $\Gamma_{X'\setminus \pi^{-1}(x)}(\CG')^\Gamma$. It would be interesting to establish such a uniformization statement in our more general context. 

\subsection{The Verlinde formula}
In \cite[Conj. 3.7]{PRq}, a finite-dimensional \emph{space of conformal blocks}  $\RH^0(\Bun_\CG, \CL)$ is introduced for certain line bundles on $\Bun_\CG$, and the question is raised whether there is an explicit expression for its dimension (a \emph{Verlinde formula}).  In the context of \cite{HoKu}, such a formula is given in \cite{HoKu2} and in \cite{Muk}. It would be interesting to give such a formula in general, starting  from an explicit presentation of the parahoric BT group scheme $\CG$ via a reductive group scheme over a Galois cover of $X$.

\subsection{The Tamagawa number}

Assume  the ground field $k$ to be finite. Let $\tau(G)$ be the Tamagawa number of $G$. Recall Weil's conjecture that $\tau(G)=1$, comp. \cite{Ko}.  It is known to hold if $G$ is split (\cite{HAnn}), and even when $G$ is quasi-split. A proof has been announced in general by  Gaitsgory-Lurie \cite{GaLurie}. Here we are interested in a relative set-up and hope to connect a general $G$ to the quasi-split case. More precisely, we assume that   there exists a tame Galois cover $X'/X$ with Galois group $\Gamma$  such that the fixed parahoric $\CG$ is of the form $\CG=\Res_{X'/X}(\CG')^\Gamma$, where $\CG'$ is a reductive group scheme over $X'$. Then $G'$ is quasi-split, cf. Proposition \ref{Hasse}. It might be fruitful to  compare the Tamagawa numbers of $G$ and $G'$ by relating the volumes of $\CG$ and $\CG'$. Note that this is different from the Langlands-Kottwitz approach \cite{Ko}, in which one compares the Tamagawa numbers of $G$ and its quasi-split inner form.  In relation to this, one might also wonder about the existence of a suitable invariant (an ``equivariant $(\CG',\Gamma)$-Tamagawa number'')   for $[\CG'/\Gamma]$ over the orbifold curve $[X'/\Gamma]$ that would facilitate the comparison.

   \bigskip
   
   \bigskip

\section{Appendix: On the Hasse Principle in the function field case} 

\centerline{\sl by Brian Conrad}

\medskip

\medskip

A connected semisimple group $G$ over a global field $k$ is said to 
{\em satisfy the Hasse Principle} 
if the natural map of sets
$$\rho_G: {\rm{H}}^1(k,G) \to \prod_v {\rm{H}}^1(k_v,G)$$
is injective. 
It is a fundamental result due to Kneser \cite{kneser}, Harder (\cite{hardersc1}, \cite{hardersc2}), and Chernousov \cite{chernousov}
over number fields, and Harder \cite[Satz A]{harder} over global function fields that any simply connected $G$ satisfies the Hasse Principle
(see \cite[Thm.\,6.6]{PR} for an exposition of the complete proof over number fields, and \cite{skip} for a wider context). 

In this appendix, we explain how that implies the Hasse Principle for $G$ which are either of adjoint type 
or absolutely simple  over {\em global function fields}, a case not as readily found
in the literature as number fields.
These facts are well-known to experts. (See Remark \ref{numfields} for the relation with arguments over number fields. 
We deduce the absolutely simple case from the adjoint-type case, whereas over number fields the proof in the literature proceeds in the opposite direction.)

Our main aim is to give a proof of: 

\begin{theorem}\label{adjoint}
If $G$ is of adjoint type over a global function field $k$ then it satisfies the Hasse Principle.
\end{theorem}

To begin the proof of Theorem \ref{adjoint}, let $q:\widetilde{G} \to G$ be the simply connected central cover and 
 $\mu = Z_{\widetilde{G}}$ be the (scheme-theoretic) center of $\widetilde{G}$. We have a short exact sequence of $k$-group schemes
\begin{equation}\label{scseq}
1 \to \mu \to \widetilde{G} \to G \to 1
\end{equation}
for the fppf topology, where $\mu$ is a finite $k$-group scheme of multiplicative type.  We will proceed in two steps:

\medskip
\noindent
{\bf Step 1}:  Let $G$ be a semi-simple group scheme over $k$ such that 
$$
\Sha^2(k, \mu) := \ker({\rm{H}}^2(k, \mu) \to \prod_v {\rm{H}}^2(k_v, \mu))
$$
vanishes. Then $G$ satisfies  the Hasse Principle. 

This step {\em does not use} the hypothesis that $G$ is of adjoint type (it works for any connected semisimple $k$-group $G$).

\smallskip 

\noindent
{\bf Step 2}:  Let $G$ be adjoint. Then $\Sha^2(k, \mu)=0$.

\smallskip

First we carry out Step 1. Pick $\xi \in {\rm{H}}^1(k,G)$, and consider any $\xi'$ in the $\rho_G$-fiber through $\xi$.  We want to show $\xi' = \xi$.
By applying the standard twisting argument that replaces $G$ with its twist by $\xi$, we reduce to the case $\xi = 1$.  In other words,
it suffices to show $\Sha^1(k,G)=1$.
We have ${\rm{H}}^1(k_v, \widetilde{G}) = 1$ for every place $v$ of $k$ since $k_v$ is non-archimedean (this vanishing is due to Kneser and Bruhat-Tits \cite[Thm.\,4.7(ii)]{bruhattits}).
Hence, ${\rm{H}}^1(k, \widetilde{G}) = 1$ by the known Hasse Principle for $\widetilde{G}$,
so it suffices to show that any element $\xi' \in \Sha^1(k,G) \subset {\rm{H}}^1(k,G)$ 
comes from ${\rm{H}}^1(k, \widetilde{G})$. 

 The short exact sequence (\ref{scseq})  gives rise to an exact sequence of pointed fppf cohomology sets
$${\rm{H}}^1(k, \widetilde{G}) \to {\rm{H}}^1(k, G) \stackrel{\delta}{\to} \check{{\rm{H}}}^2(k, \mu).$$
and likewise compatibly with $k$ replaced by any field extension (such as the completions $k_v$). 
Thus, $\Sha^1(k,G)$ is carried by $\delta$ into 
$$\check{\Sha}^2(k,\mu) := \ker(\check{{\rm{H}}}^2(k, \mu) \to \prod_v \check{{\rm{H}}}^2(k_v,\mu)),$$
so it suffices to show $\check{\Sha}^2(k,\mu) = 1$.

The natural edge map from \u{C}ech-${\rm{H}}^2$ to derived ${\rm{H}}^2$ for abelian fppf sheaves on any scheme is injective,
and is also compatible with pullback maps on the two sides (relative to any scheme map), so $\check{\Sha}^2(k,\mu)$ is naturally a subgroup of
$$\Sha^2(k,\mu) := \ker({\rm{H}}^2(k,\mu) \to \prod_v {\rm{H}}^2(k_v,\mu)).$$
 Therefore, the vanishing of $\Sha^2(k,\mu)$ implies the vanishing of $\check{\Sha}^2(k,\mu)$. This finishes Step 1. 

Now we turn to Step 2.   As for any connected reductive group over any field, $G$ has a (unique up to isomorphism) quasi-split
inner form, and passing to an inner form doesn't affect the scheme-theoretic center (by the very meaning of ``inner form'').  Such a form is also of adjoint type, so for our purposes now we can assume $G$ is quasi-split. 
 Let $B$ be a Borel $k$-subgroup of $G$
and $T$ a maximal $k$-torus in $B$ (these tori are all $k$-isomorphic, being $B/\mathscr{R}_u(B)$).

 If $(G_0, T_0, B_0)$ 
is the corresponding split form of $G$ (also of adjoint type), the Galois-twisting by
which the {\em quasi-split} $G$ is made from $G_0$ involves the finite group of diagram automorphisms acting on 
 $T_0 \subset G_0$ through a permutation action on the basis of ${\rm{X}}^*(T_0)$ consisting of simple positive roots relative to $B_0$.
Hence, $T$ is induced, so ${\rm{H}}^1(k,T)=1$.  

The preimage $\widetilde{T} \subset \widetilde{G}$ of $T$ under $q:\widetilde{G} \to G$
is a maximal torus in the preimage $\widetilde{B}  = q^{-1}(B) \subset \widetilde{G}$ that is a Borel $k$-subgroup,
and so $\widetilde{T}$ is also induced (because the {\em quasi-split} simply connected $\widetilde{G}$ is made from
the split simply connected $\widetilde{G}_0$ through Galois-twisting by a finite group of diagram automorphisms acting on
$\widetilde{T}_0 \subset \widetilde{G}_0$ through a permutation action on the basis of ${\rm{X}}_*(\widetilde{T}_0)$
consisting of simple positive coroots relative to $\widetilde{B}_0$). 
It follows that $\Sha^2(k,\widetilde{T}) = 1$ by Shapiro's Lemma and the global-to-local exact sequence of Brauer groups in class
field theory for finite separable extensions of $k$.

From the diagram of $k$-group schemes
$$1 \to \mu \to \widetilde{T} \to T \to 1$$
that is short exact for the fppf topology, 
we get a long exact sequence of fppf cohomology groups
$${\rm{H}}^1(k,T) \to {\rm{H}}^2(k,\mu) \to {\rm{H}}^2(k,\widetilde{T})$$
and likewise with each $k_v$ in place of $k$.
(Since $T$ and $\widetilde{T}$ are smooth, a result of Grothendieck \cite[Thm.\,11.7]{brauer} ensures that the fppf cohomology of these tori
in each positive degree coincides with the Galois cohomology for these 
tori in that same degree over any field extension $K$ of $k$. Thus, there is no ambiguity about the 
cohomology of tori appearing here.)    Hence, the above vanishing properties
for $T$ and $\widetilde{T}$ give an injection of $\Sha^2(k,\mu)$ into $\Sha^2(k, \widetilde{T}) = 1$, so we are done.

 \begin{corollary}\label{abssimple} For absolutely simple $G$ over a global function field $k$, the Hasse Principle holds.\footnote{Semisimple but not absolutely simple counterexamples to the Hasse principle are given over every number field by Serre 
\cite[Ch.\,III, \S4.7]{serre}.}
 \end{corollary}
 
 \begin{proof}
 Since the Hasse Principle is known for $G$ that are either simply connected or adjoint type, 
it remains to treat the cases when $G$ is neither simply connected nor adjoint type.  This
makes $G$ of type A$_n$ with $n \ge 3$ or D$_n$ with $n \ge 4$ and $\mu_G := \ker(\widetilde{G} \to G)$ a
{\em nontrivial proper} $k$-subgroup scheme of
$Z_{\widetilde{G}}$.  By Step 1 in the proof of Theorem \ref{adjoint} (which applies
to all connected semisimple $k$-groups), it suffices to show $\Sha^2(k, \mu_G)=1$. 
 
For type D, necessarily $\mu_G = \mu_2$ (as $\mu_2$ has no non-trivial Galois-twists)
and $\Sha^2(k, \mu_m)=1$ for all $m$ by the Brauer group argument as near the end of the proof of Step 2   in the proof of Theorem \ref{adjoint}. 
For type A$_n$ (with $n \ge 3$) necessarily $\mu_G$ is either $\mu_d$ for some nontrivial proper divisor $d$ of $n+1$
or a {\em quadratic} Galois twist $\mu_d(\chi)$ through inversion.  It remains to show $\Sha^2(k, \mu_d(\chi)) = 1$.
If $d=2$ then $\mu_d(\chi) = \mu_d$, so we may assume $d > 2$.  Then $\mu_d(\chi)$ is the center $\mu_{G'}$ of 
an outer form $G'$ of ${\rm{SL}}_d$, such as a special unitary group ${\rm{SU}}_d(K/k)$ associated to a quadratic Galois extension $K/k$.
But then $\Sha^2(k, \mu_d(\chi)) = \Sha^2(k, \mu_{G'}) = 1$ by Step 2   in the proof of Theorem \ref{adjoint}.
(This trick with ${\rm{SU}}_d(K/k)$, suggested by S.\,Garibaldi, is much simpler than the original argument.) 
\end{proof}
 
\begin{remark}\label{numfields}  For the interested reader, here is a comparison of the preceding arguments with two primary references
on the Hasse Principle over number fields beyond the simply connected case: the book \cite{PR} and the paper \cite{sansuc}. Indeed, the proofs above  adapt  to work over number fields (see below), so it is of interest to compare them with existing approaches.

Over number fields, Theorem \ref{adjoint} is proved in \cite[\S6.5, Thm.\,6.22]{PR} in a different way
using case-by-case determination of centers of various
connected semisimple groups. 

Sansuc \cite[Cor.\,5.4(i),(ii)]{sansuc} also proves both Theorem \ref{adjoint} and Corollary \ref{abssimple}
over number fields (the E$_8$-avoidance there became unnecessary via later work
of Chernousov \cite{chernousov}) but with several key differences (especially in the presence of real places).  Firstly, the analogue of the argument in  Step 1 
is simpler over number fields, since $k$-groups of finite type are smooth when ${\rm{char}}(k)=0$ (so the distinction between
$\check{{\rm{H}}}^i(k,C)$ and ${\rm{H}}^i(k,C)$ for commutative $k$-groups $C$ of finite type does not arise). On the other hand, the presence of archimedean places necessitates the additional fact that the map
\begin{equation}\label{surj}
{\rm{H}}^1(k, \mu) \to \prod_{v|\infty} {\rm{H}}^1(k_v, \mu)
\end{equation}
is surjective. This follows easily from real approximation for tori (and is valid for any diagonalizable group $\mu$). 

Secondly,  in \cite{sansuc} the logic
of implication between the results is the other way around, putting the emphasis on directly proving Corollary \ref{abssimple}.  
Finally, passage to a quasi-split inner form is used in a rather different way:  instead of the induced property of tori of Borel $k$-subgroups (for adjoint-type quasi-split
$G$), the ``meta-cyclic'' property of the splitting field of the \'etale Cartier dual of $\mu_G$ (for absolutely simple quasi-split $G$) is invoked
(see \cite[Cor.\,5.2]{sansuc} and the end of the proof of \cite[Cor.\,5.4]{sansuc}). That meta-cyclic property rests on knowledge of the small automorphism groups of connected
Dynkin diagrams.

Note here that Sansuc actually shows bijections \cite[Cor. 4.4, (4.3.2)]{sansuc} 
\begin{equation}\label{bijection}
\ker \rho_G\simeq \Sha^2(k, \mu)\simeq\Sha^1(k, {\rm{X}}^*(\mu))^*
\end{equation}
(the E$_8$-avoidance there became unnecessary via later work
of Chernousov). The second isomorphism is given by Tate global duality for finite Galois modules over number fields \cite[Ch.\,VIII, Thm.\,8.6.9]{neukirch}. (A somewhat different proof of (\ref{bijection}) is also provided by Kottwitz \cite[(1.8.4), (4.2.2)]{kottwitz}.)

Our proof of Theorem \ref{adjoint} adapts to work over number fields, as follows. Using (\ref{bijection})  (in fact, only the simpler injectivity of the first map is needed), it is sufficient to prove the vanishing of $\Sha^2(k, \mu)$  (and the surjectivity of \eqref{surj}),
completing Step 1 over number fields. Step 2 works as written over all global fields, as does the deduction of Corollary \ref{abssimple} 
from Theorem \ref{adjoint}.
\end{remark}

\begin{remark}
Here are some additional remarks on the literature. First, Step 1 also appears (for number fields) in the  lecture notes of Kneser \cite[\S 5.2]{kneser} and in the survey of Harder \cite[Satz 4.3.2]{harder1}. The argument in Step 2 also appears (with a reference to another paper of Harder) in \cite[\S 2, Bemerkung, p.~129]{harder} (the fact that this section treats groups of type A plays no role at this point). As further more recent references in the function field case, we mention  \v Cesnavi\v cius \cite{Ce}, Th\' ang \cite{Th}, and Rosengarten \cite{Ro}.
\end{remark}


\begin{thebibliography}{AAA}
 
  \bibitem{An} J.\,Ansch{\"u}tz, \textit{Extending torsors on the punctured $\Spec(A_{{\rm inf}})$,}  J. Reine Angew. Math. 783 (2022), 227--268.

 \bibitem{AW} M.\,Atiyah, C.T.C\,Wall, \textit{Cohomology of groups}. Algebraic Number Theory (Proc. Instructional Conf., Brighton, 1965), 94--115, Thompson, Washington, D.C., 1967.
 
 \bibitem{BalaS} V.\,Balaji, C.\,S.\,Seshadri, \textit{ Moduli of parahoric $\CG$-torsors on a compact Riemann surface.} J. Algebraic Geom. 24 (2015), no. 1, 1--49.
  
 \bibitem{Bro} M.\,Broshi, \textit{$G$-torsors over a Dedekind scheme.} J. Pure Appl. Algebra 217 (2013), no. 1, 11--19.
 
 \bibitem{BTII} F.\,Bruhat, J.\,Tits, \textit{Groupes r\'eductifs sur un corps local. II. Sch\'emas en groupes. Existence d'une donn\'ee radicielle valu\'ee.} Publ. Math. IHES, 60 (1984), 197--376.
 
  \bibitem{BTIII} F.\,Bruhat, J.\,Tits, \textit{Groupes r\'eductifs sur un corps local. III. Compl\'ements et applications \`a la cohomologie galoisienne}
J. Fac. Sci. Univ. Tokyo Sect. IA Math. 34 (1987), no. 3, 671--698.
 
  \bibitem{BTcla} F.\,Bruhat, J.\,Tits, \textit{Sch\'emas en groupes et immeubles des groupes classiques sur un corps local. II. Groupes unitaires.}   Bull. Soc. Math. France 115 (1987), no. 2, 141--195. 
 
 \bibitem{CaraLev}  A.\,Caraiani, B.\,Levin, \textit{Kisin modules with descent data and parahoric local models.} Ann. Sci. Ec. Norm. Sup\'er. (4) 51 (2018), no. 1, 181--213. 
 
 \bibitem{Caraetc}  A.\,Caraiani, M.\,Emerton, T.\,Gee, D.\,Savitt,  \textit{Local geometry of moduli stacks of two-dimensional Galois representations}, arXiv:2207.14337

 \bibitem{Conrad} B.\,Conrad, \textit{Reductive group schemes}, Autour des sch\'emas en groupes. Vol. I, 93--444, Panor. Synth\`eses, 42/43, Soc. Math. France, Paris, 2014. 
 
\bibitem{Cotner} S.\,Cotner, \textit{Morphisms of character varieties}, arXiv:2304.10135

 
  \bibitem{Dam1} C.\,Damiolini, \textit{Conformal blocks attached to twisted groups}, Mathematische Zeit.  295 (2020), no. 3-4, 1643--1681. 
 
 \bibitem{Dam} C.\,Damiolini, \textit{ 
On equivariant bundles and their moduli spaces}, arXiv:2109.08698

\bibitem{DH} C.\,Damiolini, J.\,Hong,\textit{ Local types of $(\Gamma, G)$-bundles and parahoric group schemes,
} Proc. London Math. Soc.(3) 27 (2023), 261--294.


 \bibitem{Muk} T.\,Deshpande, S.\,Mukhopadhyay, \textit{Crossed modular categories and the Verlinde formula for twisted conformal blocks}, arXiv:1909.10799


\bibitem{Douai} J.-C.\,Douai,
\textit{Suites exactes d\'eduites de la suite spectrale de Leray en cohomologie non ab\'elienne.} 
J. Algebra 79 (1982), no. 1, 68--77. 

 \bibitem{Edix} B.\,Edixhoven, \textit{ N\'eron models and tame ramification.} Compositio Math. 81 (1992), no. 3, 291--306.
 
\bibitem{GaLurie} D.\,Gaitsgory, J.\,Lurie, \textit{Weil's conjecture for function fields. Vol. 1.} Annals of Mathematics Studies, 199. Princeton University Press, Princeton, NJ, 2019. viii+311 pp.
 
 \bibitem{Gi} P.\,Gille, \textit{Torseurs sur la droite affine.} Transform. Groups 7 (2002), no. 3, 231--245.
 
\bibitem{Giraud} J.\,Giraud, \textit{Cohomologie non ab\'elienne}. Grundlehren der mathematischen
Wissenschaften No. 179, Springer-Verlag, New York/Berlin, 1971.

\bibitem{Gross} B.\,Gross, \textit{Parahorics}, available at http://www.math.harvard.edu/$\sim$gross/eprints.html 

\bibitem{GrothBour} A.\,Grothendieck, \textit{Sur le m\'emoire de Weil: g\'en\'eralisation des fonctions ab\'eliennes}. 
S\'eminaire Bourbaki, Vol. 4, Exp. No. 141 (1956-57), 57--71, Soc. Math. France, Paris, 1995. 


\bibitem{Toho} A.\,Grothendieck, \textit{Sur quelques points de  d'alg\`ebre homologique}.  Tohoku Math. J. (2) 9 (1957), 119--221.

 \bibitem{HInv} G.\,Harder, \textit{Halbeinfache Gruppenschemata \"uber Dedekindringen.} Invent. Math. 4 (1967), 165--191. 
 
 \bibitem{HAnn} G.\,Harder, \textit{Chevalley groups over function fields and automorphic forms.} Ann. of Math. (2) 100 (1974), 249--306. 
 
  \bibitem{HCrel} G.\,Harder,  \textit{\"Uber die Galoiskohomologie halbeinfacher algebraischer Gruppen. III. } J. Reine Angew. Math. 274(275) (1975), 125--138.
  
 \bibitem{Hein} J.\,Heinloth, \textit{Uniformization of $\CG$-bundles}. Math. Ann. 347 (2010), no. 3, 499--528.
 
 \bibitem{HoKu} J.\,Hong, S.\,Kumar,    
\textit{Conformal blocks for Galois covers of algebraic curves.} Compos. Math. 159 (2023), no. 10, 2191--2259.

 \bibitem{HoKu2} J.\,Hong, S.\,Kumar, \textit{Twisted conformal blocks and their dimension.}
 arXiv:2207.09578

 

\bibitem{Ko} R.\,Kottwitz, \textit{Tamagawa numbers}. Ann. of Math. (2) 127 (1988), no. 3, 629--646.

 \bibitem{Lar} M.\,Larsen, \textit{Maximality of Galois actions for compatible systems}. Duke Math. J. 80 (1995), 601--630.
 
 
 
  \bibitem{PRtwisted} G.\,Pappas, M.\,Rapoport, \textit{Twisted loop groups and their affine flag varieties}. Adv. Math.
  219 (2008), no.~1, 118--198. With an appendix by Th. Haines and Rapoport.
 
 
 \bibitem{PRq} G.~Pappas, M.\,Rapoport, \textit{Some questions about $\CG$-bundles on curves.} Algebraic and arithmetic structures of moduli spaces (Sapporo 2007), 159--171, Adv. Stud. Pure Math., 58, Math. Soc. Japan, Tokyo, 2010.
 
 \bibitem{PRglsv}  G.\,Pappas, M.~Rapoport, \textit{$p$-adic shtukas and the theory of global and local Shimura varieties}, arXiv:2106.08270
 
 \bibitem{PRintls} 
G.\,Pappas, M.\,Rapoport, \textit{On integral local Shimura varieties}, arXiv:2204.02829 

 \bibitem{PZ} G.~Pappas, X.~Zhu, \textit{Local models of Shimura varieties and a conjecture of Kottwitz}. Invent. math. {194} (2013), 147--254.
 
 
 
  \bibitem{PraYu} G.\,Prasad, J.-K.\,Yu, \textit{On finite group actions on reductive groups and buildings.} Invent. Math. 147 (2002), no. 3, 545--560.
 
\bibitem{Raynaud} M.\,Raynaud, \textit{Anneaux Locaux Hens\'eliens}, Lecture Notes in Math. 169, Springer-Verlag, New York, 1970. 
 
 
 \bibitem{Serre} J.-P.\,Serre, \textit{Galois cohomology.} Translated from the French by Patrick Ion and revised by the author. Corrected reprint of the 1997 English edition. Springer Monographs in Mathematics. Springer-Verlag, Berlin, 2002. x+210 pp.
  

 
  \bibitem{Stein} R.\,Steinberg, \textit{Endomorphisms of linear algebraic groups.} Memoirs of the American Mathematical Society, No. 80. American Mathematical Society, Providence, R.I. 1968 
  
   \bibitem{Weil} A.\,Weil, \textit{G\'en\'eralisation des fonctions ab\'eliennes}. J. Math. pures et appl. 17 (1938), 47--87.
   
   \end{thebibliography}

\medskip
 
   
   
   \bigskip
   
 

   
   \begin{thebibliography}{aaaaa}
   
\bibitem[BTIII]{bruhattits} F.\,Bruhat, J.\,Tits, {\em Groupes alg\'ebriques sur un corps local. Chapitre III.  Compl\'ements et
applications \`a la cohomologie galoisienne}, J.\,Fac.\,Sci.\,Univ.\,Tokyo Sect.\,IA Math.\,{\bf 34}, no.\,3 (1987), pp.\,671-698.

\bibitem[Ce]{Ce}  K.\,\v Cesnavi\v cius, {\em Poitou-Tate without restrictions on the order}, Math. Res. Lett. 22 (2015), no. 6, 1621-1666.

\bibitem[Ch]{chernousov} V.\,Chernousov, {\em The Hasse Principle for groups of type ${\rm{E}}_8$}, Soviet Mathematics Doklady, Vol.\,39, No.\,3 (1989), pp.\,592-596.

\bibitem[Ga]{skip} S.\,Garibaldi, {\em What is a Linear Algebraic Group?}, Notices of the AMS, Volume 57, No.\,9 (October 2010), pp.\,1125-6.

\bibitem[Gr]{brauer} A.\,Grothendieck, {\em Le groupe de Brauer III: exemples et compl\'ements},
in ``Dix Expos\'es sur la cohomologie des sch\'emas'', North-Holland, Amsterdam, 
1968, pp.\,88--188.

\bibitem[Ha1]{hardersc1} G.\,Harder, {\em \"{U}ber die Galoiskohomologie halbeinfacher Matrizengruppen I}, Mathematische Zeitschrift
{\bf 90} (1965), pp.\,404-428.

\bibitem[Ha2]{hardersc2} G.\,Harder, {\em \"{U}ber die Galoiskohomologie halbeinfacher Matrizengruppen II}, Mathematische Zeitschrift
{\bf 92} (1966), pp.\,396-415.

\bibitem[Ha3]{harder} G.\,Harder, {\em \"{U}ber die Galoiskohomologie halbeinfacher algebraischer Gruppen III}, J.\,Reine Angew. Math., 
{\bf 274/5} (1975), pp.\,125-138.

\bibitem[Ha4]{harder1} G.\,Harder, {\em Bericht \"uber neuere Resultate der Galoiskohomologie halbeinfacher Gruppen.} Jber. Deutsch. Math.-Verein. 70 (1967/68), Heft 4, Abt. 1, 182-216. 

\bibitem[Kn]{kneser} M.\,Kneser, {\em Lectures on Galois cohomology of classical groups} (notes by P.\,Jothilingam), 
TIFR Lectures Notes in Mathematics {\bf 47},
TIFR Bombay, 1969.

\bibitem[Ko]{kottwitz} R.\,Kottwitz, {\em Stable Trace Formula: Cuspidal Tempered Terms}, Duke Math Journal, Vol.\,51, No.\,3 (1984), pp.\,611-650.

\bibitem[NSW]{neukirch} J.\,Neukirch, A.\,Schmidt, K.\,Wingberg, {\em Cohomology of Number Fields} (2nd ed.), Grundlehren der 
mathematischen Wissenschaftten {\bf 323}, Springer-Verlag, New York (2008).

\bibitem[PR]{PR} V.\,Platonov, A.\,Rapinchuk, {\em Algebraic Groups and Number Theory}, Pure and Applied Mathematics
{\bf 139}, Academic Press, New York, 1994.

\bibitem[Ro]{Ro} Z.\,Rosengarten, {\em Tamagawa numbers and other invariants of pseudo-reductive groups over global function fields.}
Algebra Number Theory 15 (2021), no.8, pp.\,1865-1920.

\bibitem[Sa]{sansuc} J.-J.\,Sansuc, {\em Groupe de Brauer et arithm\'etique des groupes alg\'ebriques lin\'eaires sur un corps de nombres}, 
J.\,Reine Angew. Math., {\bf 327} (1981), pp.\,12-80.

 \bibitem[Se]{serre} J.-P.\,Serre, {\em Galois Cohomology}, Springer Monographs in Mathematics, Springer-Verlag, 1997. 

\bibitem[Th]{Th}  N.\,Q.~Th\' ang, {\em On Galois cohomology of semisimple groups over local and global fields of positive characteristic},
Math. Z. 259 (2008), no. 2, 457-467.
   
   
   
  \end{thebibliography}
 \end{document}